\documentclass[10pt, english, a4paper]{amsart}

\usepackage[english]{babel}
\usepackage[T1]{fontenc}
\usepackage{tikz}
    \usepackage{amssymb}
 \usepackage{xcolor} 
\usepackage[all]{xy}
\usepackage{graphicx}

\usepackage{fancyhdr}

\usepackage{hyperref} 

\font\Times=ptmr at 10pt

\font\it=ptmri at 10pt
\font\smallit=ptmri at 8 pt

\Times

 \theoremstyle{plain}
 \newtheorem{lemma}{Lemma}[section]
\newtheorem{theorem}[lemma]{Theorem }
\newtheorem{corollary}[lemma]{Corollary}
\newtheorem{prop}[lemma]{Proposition}

\theoremstyle{definition}
\newtheorem{example}{Example}
\newtheorem{definition}[lemma]{Definition }

\newtheorem{rmk}[lemma]{Remark}
\newtheorem{rmks}[lemma]{Remarks}

\long\def\[#1\]{\ignore\recognize}

\def\ltextindent#1{\hbox to \hangindent{#1\hss}\ignorespaces}
\long\def\ignore#1\recognize{}

\def\med{\medskip}
\def\hs{\hskip}
\def\vs{\vskip}

\def\ds{\displaystyle}

\def\ol{\overline}

\def\wh{\widehat}
\def\wt{\widetilde}
\def\sm{\setminus}

\def\{{\lbrace}
\def\}{\rbrace}

\def\isom{\cong}

\def\map{\rightarrow}

\def\mapname#1{\,\,\smash{\mathop{\longrightarrow}\limits^{#1}}\,\,}

\def\6{\partial}
\def\1{1\!\! 1}

\def\R{\Bbb R}
\def\C{{\Bbb C}}
\def\N{{\Bbb N}}

\def\A{\Bbb A}
\def\P{{\mathcal P}}

\def\O{{\cal O}}
\def\ord{{\rm ord}}

\def\abs#1{\vert#1\vert}
\def\b{\beta}
\def\a{\alpha}

\def\phi{\varphi}

\def\Reg{{\rm Arq.Reg}}
\def\Sing{{\rm Arq.Sing}}
\def\Spec{{\rm Spec}}
\def\arqregular{arq.regular }
\def\arqregulaR{arq.regular}
\def\Arqregular{Arq.regular }
\def\arqregularity{arq.regularity }
\def\arqsingular{arq.singular }

\def\Ker{{\rm Ker}}

\def\Gl{{\rm Gl}}
\def\height{{\rm ht}}

\def\x{{\rm x}}
\def\y{{\rm y}}
\def\z{{\rm z}}
\def\u{{\rm u}}
\def\v{{\rm v}}
\def\w{{\rm w}}

\def\uu{u_\infty}
\def\dd{{\bf d}}

\def\VV{{\mathcal V}}
\def\UU{{\mathcal U}}

\def\YY{{\mathcal Y}}
\def\ZZ{{\mathcal Z}}

\def\SS{{\mathcal S}}

\def\O{{\mathcal O}}
\def\PP{{\mathcal P}}
\def\RR{{\mathcal R}}

\def\GG{{\mathcal G}}

\def\AA{{\mathcal A}}
\def\AAA{\mathcal A_\circ}
\def\AACC{\mathcal A}
\def\BB{{\mathcal B}}
\def\BBB{\mathcal B_\circ}

\def\DD{\Delta}

\def\WW{{\mathcal{W}}}

\def\ff{{f_\infty}}

\def\vv{v}
\def\ww{w}
\def\mm{{ M}}

\def\zz{z}

\def\aa{{\rm a}}
\def\t{{\rm t}}
\def\s{{\rm s}}
\def\k{{\bf k_\circ}}
\def\kkk{{\bf k}}

\def\kk{k}

\def\mm{M}
\def\nn{N}

\def\diag{{\rm diag}}

\def\an{an }

\def\arquile{arquile }
\def\Arquile{Arquile }

\def\ttextile{textile }

\def\langlew{\longrightarrow}

\ignore

\begin{xy}
\xymatrix{
\text{geometrisch:} & & \text{algebraisch:} & \\
\A_K^m \ar[d]_\pi & X \ar[l]_j \ar[dl]^{\tau := \pi|_X} & R = \faktor{K[x_1, \dots, x_n]}{I} & K[x_1, \dots, x_m] \ar[l]_{\qquad j^*} \\
\A_K^n & & S = K[z_1, \dots, z_n] \ar[ur]_{\pi^*} \ar[u]^{\tau^*}
}
\end{xy}

\recognize

\usepackage{lipsum}

\newcommand\blfootnote[1]{%
  \begingroup
  \renewcommand\thefootnote{}\footnote{ #1}%
  \addtocounter{footnote}{-1}%
  \endgroup
}

\begin{document}


\title{Arquile Varieties {\bf -} Varieties Consisting of Power Series in a Single Variable}

\author{Herwig Hauser, Sebastian Woblistin}

\blfootnote{ 

The authors acknowledge the support by the Austrian Science Fund FWF through the projects P-25652, P-31338 and AI-0038211, respectively through the joint project I-1776 of FWF and the French National Research Agency ANR. Part of the work has been done during the special semester on Artin approximation at CIRM in spring 2015 within the Chaire Jean Morlet of the first-named author at Aix-Marseille University. The hospitality of the personnel at CIRM was greatly appreciated.}

\maketitle


\pagestyle{myheadings}
\markright{\hfill ARQUILE VARIETIES\hfill}


\begin{abstract} Spaces of power series solutions $y(\t)$ in one variable $\t$ of systems of polynomial, algebraic, analytic or formal equations $f(\t,\y)=0$ can be viewed  as ``infinite-dimensional'' varieties over the ground field $\kkk$ as well as ``finite-dimensional'' schemes over the power series ring $\kkk[[\t]]$. We propose to call these solution spaces {\smallit arquile varieties}, as an enhancement of the concept of {\smallit arc spaces}. It will be proven that arquile varieties admit a natural stratification $\YY=\bigsqcup\YY_d$, $d\in\N$, such that each stratum $\YY_d$ is isomorphic to a cartesian product $\ZZ_d\times \A^\infty_\kkk$ of a finite dimensional, possibly singular variety $\ZZ_d$ over $\kkk$ with an affine space $\A^\infty_\kkk$ of infinite dimension. This shows that the singularities of the solution space of $f(\t,\y)=0$ are confined, up to the stratification, to the finite dimensional part.

Our results are established simultaneously for algebraic, convergent and formal power series, as well as convergent power series with prescribed radius of convergence. The key technical tool is a linearization theorem, already used implicitly by Greenberg and Artin, showing that analytic maps between power series spaces can be {\smallit essentially} linearized by automorphisms of the source space.

Instead of stratifying arquile varieties one may alternatively consider formal neighborhoods of their regular points and reprove with similar methods the Grinberg-Kazhdan-Drinfeld factorization theorem for arc spaces in the classical and in the more general setting.

\end{abstract}


\tableofcontents






\section{Concepts and results}\label{section_concepts}

Throughout this article, we will work simultaneously with the ring $\kkk[[\t]]$ of formal power series  in one variable $\t$ over a perfect field $\kkk$, the ring $\kkk\{\t\}$ of convergent power series over a valued field $\kkk$, the ring $\kkk\{\t\}_s$ of convergent power series with finite $s$-norm, for $s>0$, or the Henselian ring  $\kkk\langle\t\rangle$ of algebraic (= Nash) series, the algebraic closure of $\kkk[\t]$ inside $\kkk[[\t]]$. We will reserve the letter $\AA$ for any of these rings.%
\footnote{ The ring $\kkk\{\t\}_s=\{\sum\, c_k\t^k\in \kkk\{\t\},\, \sum\abs{c_k}s^k<\infty\}$ is considered as a Banach-algebra with the norm $\abs{\sum\, c_k\t^k}_s =\sum\abs{c_k}s^k$, for every $s>0$, as described e.g. in \cite{GR}.} 
The subrings of series $y(\t)$ without constant term, say, with $y(0)=0$, will be denoted by $\AAA$ or, in case of ambiguity, by $\k[[\t]]$, $\k\{\t\}$, $\k\{t\}_s$, respectively $\k\langle\t\rangle$. These rings will be equipped with the $\t$-adic topology given by the ideals $\langle \t\rangle^r$ as a basis of neighborhoods of $0$.%
\footnote{ The restriction to series without constant term has mostly notational reasons, since for these the substitution of power series into each other poses no problems. For series with non-zero constant term one would have to be more cautious when defining substitution, adding each time the appropriate assumptions. Note that $\AAA$ is a ring without $1$-element.}
We distinguish the various settings (formal, convergent, ...) by saying that the involved series $y(\t)$ have the respective {\it quality}.\medskip

To formulate our results, we need to introduce some basic concepts about power series spaces and maps between them.\medskip

Let $\AAA^m$ denote the $m$-fold cartesian product of $\AAA$. Its elements are vectors of power series $y(\t)=(y_1(\t),...,y_m(\t))$ vanishing at $0$, commonly called ($m$-fold) {\it arcs centered at} $0$. As each series $y_i(\t)$ is given by the coefficients $(y_{ij})_{j\geq 1}$ in $\kkk$ of its power series expansion $\sum_{j\geq 1} y_{ij}\cdot \t^j$, we may identify $\AAA^m$ with a subspace of the space $(\kkk^{\N_{>0}})^m$ of vectors of sequences in $\kkk$ indexed by the positive integers. The space $\AAA^m$ is the ambient space where our objects of interest live: These are the collections of all power series solutions $y(\t)$ to formal, analytic or algebraic equations $f(\t,y(\t))=0$. We propose to call $\AAA^m$ the $m$-fold {\it affine space} over $\AAA$. It will carry three topologies, the $\t$-adic topology induced from $\AAA$, and the textile, respectively \arquile topologies defined later on. For $\AAA=\kkk[[\t]]$, its {\it coordinate ring} is the polynomial ring $\kkk[\y_{ij},\,  1\leq i\leq m,\, j\geq 1]$ in countably many variables $\y_{ij}$. We abbreviate this ring by $\kkk[\y_{m,\infty}]$.%
\footnote{ Variables $\t$, $\y$ and $\y_{ij}$ are denoted by roman characters, power series $y=y(\t)$ and their coefficients $y_{ij}$ by slanted characters.}
\medskip

A map $\tau:\UU\subset\AAA^m\map\AA^\kk$ between (subsets of) affine  spaces is given by prescribing the coefficients of the images $\tau(y(\t))$ as functions of the coefficients of the power series vectors $y(\t)$. We say that $\tau$ is {\it textile} if each coefficient of $\tau(y(\t))$ is a {\it polynomial} in (finitely many of) the coefficients of the input $y(\t)$.%
\footnote{ In the setting of formal power series $\AAA=\k[[\t]]$, textile maps are therefore the restriction to the $\kkk$-points of morphisms of affine schemes if we consider $\AAA^m$ and $\AAA^\kk$ as the ``infinite-dimensional'' schemes $\Spec(\kkk[\y_{m,\infty}])$ and $\Spec(\kkk[\y_{\kk,\infty}])$ over $\kkk$.}
A priori, we do not prescribe any further relations between these coefficient polynomials, even though later on they will be highly related in the specific  applications we have in mind. Note that textile maps are $\t$-adically continuous. In the setting of convergent or algebraic power series, the coefficient polynomials have to be modelled such that the image power series vectors are again convergent, respectively, algebraic.\medskip

Zerosets $\tau^{-1}(0)$ of textile maps $\tau:\AAA^m\map\AA^\kk$ define the closed sets of a topology on $\AAA^m$, the {\it textile topology}. These sets are given by (usually infinite) systems of polynomial equations in the coefficients of the involved power series vectors. A subset $\RR$ of $\AAA^m$ is called {\it textile cofinite} if it admits a finite defining system. For $\AAA=\kkk[[\t]]$, the textile topology coincides with the Zariski-topology. Compare this with the concept of {\it Greenberg schemes} developed in \cite{CNS}. \medskip

Textile maps form a rather large class of maps between affine spaces over $\AAA$. We will essentially use two types of such maps, the $\kkk$-linear maps given by the Weierstrass division of power series, sending a series to its quotient or remainder with respect to a prescribed divisor (see the section on division), and maps given by substitution (see below). It seems that there is no substantially smaller {\it natural} class of maps between power series spaces which contains these two types. \medskip


An {\it \arquile} map $\alpha:\UU\subset\AAA^m\map\AA^\kk$ is defined by the {\it substitution} of the $\y$-variables of a given power series vector $f(\t,\y)$ by power series vectors $y(\t)\in\UU$. More precisely, if there exists, for the chosen quality of $\AAA$, a power series vector $f=(f_1,...,f_\kk)$ in $\kkk[[\t,\y]]^\kk$, respectively $\kkk\{\t,\y\}^\kk$, or $\kkk\langle\t,\y\rangle^\kk$, depending on $1+m$ variables $\t,\y_1,...,\y_m$, such that $\alpha$ is defined by
\footnote{  In the case where $\AAA=\k\{\t\}_s$ is the space of convergent series with finite $s$-norm, for some $s>0$, the map $\a$ is only defined on the set of vectors $y(\t)$ for which the norms $\abs{y_i(\t)}_s$ are sufficiently small so that $f(\t,y(\t))$ has again finite $s$-norm, see section 7 of \cite{HM} for details.}
$$\alpha(y(\t))=f(\t,y_1(\t),...,y_m(\t)).$$
\Arquile maps are textile. Actually, the definition of arquile maps appears to be somewhat too restrictive and is not the most natural one: before applying $f$, it should  be allowed to cut the power series vector $y(\t)$ at any chosen degree into the sum of a polynomial vector (its jet) and a remainder power series vector, and to define arquile maps as the sum of a polynomial map in the coefficients of the jet and a substitution map applied to the remainder series as in the definition above. This is justified since we are only interested in power series modulo polynomials, that is, up to cutting off polynomial truncations. However, the more general definition would increase enormously the notational luggage, so we prefer to work here with the more restrictive definition.\medskip

The relation between $\alpha$ and $f$ will be expressed in symbols by 
$$\alpha =f_\infty:\UU\subset\AAA^m\map{\mathcal A}^\kk,$$
$$y(\t)\map f(\t,y(\t)),$$
and by saying that $\alpha$ is the \arquile map induced by $f$. The rings $\kkk[[\t,\y]]$, $\kkk\{\t,\y\}$ and $\kkk\langle\t,\y\rangle$ will be denoted by the letter $\BB$, and $\BBB=\BB\cap\k[[\t,\y]]$ denotes the ring of series $f=f(\t,\y)$ with $f(0,0)=0$. \medskip


An {\it \arquile variety} $\YY$ is a subset of $\AAA^m$ given as the zeroset 
$$\YY=\YY(f)=\alpha^{-1}(0)=\{y(\t)\in\AAA^m,\, f(\t,y(\t))=0\}$$
of an \arquile map $\alpha=f_\infty: \AAA^m\map{\mathcal A}^\kk$. This is a generalization of the concept of the arc space of an algebraic variety, by allowing the parameter $\t$ to appear in the defining equations of $\YY$ and by considering also  formal, analytic and algebraic equations instead of just polynomial ones. See section \ref{arquile} for a detailed discussion. Such subsets form the closed sets of a topology on $\AAA^m$, the {\it \arquile} topology. Accordingly, we get the notion of open, closed and locally closed \arquile subsets of $\AAA^m$.%
\footnote{ One could also define affine \arquile schemes, but we will have no need of these here.} 
Closed arquile sets are defined by an infinite set of polynomial equations $F_\ell=0$, $\ell\in\N$, for the coefficients $y_{ij}$ of the involved power series vectors.%
\footnote{ These polynomial equations stem from the system $f(\t,y(\t))=0$ by Taylor expansion.}
Both the $\t$-adic and the textile topologies are finer than the \arquile topology.\medskip


\begin{example} To illustrate, consider the set $\YY$ of pairs of power series $(y(\t),\zz(\t))\in\AAA^2=\k[[\t]]^2$ satisfying the quadratic equation
$$f(\t,y(\t),\zz(\t))=y(\t)^2 -2\t\cdot y(\t)\cdot \zz(\t)^2+\t^4=0.$$
One may solve for $y$ in terms of $\zz$ and get
$$y(\t)=\t\cdot\zz(\t)\pm \t\cdot \sqrt{\zz(\t)^4-\t^2}.$$
The root $\sqrt{\zz(\t)^4-\t^2}$ is a formal power series in $\t$ of order $1$ for any series $\zz(\t)\in\AAA$. 
\end{example}


There is a natural way to define the {\it \arquile regular} and {\it singular loci} $\Reg(\YY)$ and $\Sing(\YY)$ of \an \arquile variety $\YY=\YY(f)$:%
\footnote{ As the regularity of an \arquile variety does not exactly match the respective definition in the finite dimensional setting, we are led to use a separate terminology. For a detailed study of this concept, see \cite{Wo}.}
Let $I=I_\YY$ be the ideal of $\BB$ of power series $f(\t,\y)$ vanishing on $\YY$, say $f(\t,y(\t))=0$ for all $y\in\YY$. For a point $y$ of $\YY$, denote by $\PP_y$ the prime ideal of $\BB$ of power series $f(\t,\y)$ vanishing at $y$,
$$\PP_y= \{f\in\BB,\, f(\t,y(\t))=0\}.$$
Let $\BB_{\PP_y}$ be the localization of $\BB$ at $\PP_y$. This is a Noetherian ring in the formal, convergent and algebraic setting.%
\footnote{ In the case $\AAA=k\{\t\}_s$, which is a non-Noetherian ring, this definition of regularity does not apply. Instead, one has to take the non-vanishing of the minors from assertion (a) of Theorem \ref{regularpoints} as the correct definition, see also Theorem \ref{singularlocus}.}
Then $\YY$ is called {\it \arquile regular} at $y$, or $y$ is called an {\it \arqregulaR} point of $\YY$, if the factor ring $$\BB_{\PP_y}/I\cdot\BB_{\PP_y}$$ is a  regular local ring. Otherwise, $\YY$ is called {\it \arquile singular} at $y$. We set $\Reg=\{y\in\YY,\, y$ \arqregular point of $\YY\}$ and $\Sing(\YY)=\YY\sm\Reg(\YY)$.\medskip\goodbreak

If $\YY=X_{\infty,0}$ is the space of arcs centered at $0$ of an algebraic subvariety $X$ of $\A^m_\kkk$, 
a point $y\in\YY$ is \arqregular if and only if the image of $y$ in $X$ (i.e., the generic point of $y$) does not lie entirely in the singular locus of $X$, that is, if $y\not\in ({\rm Sing}\, X)_{\infty,0}$. This condition appears naturally in the theory of arc spaces, see e.g.~\cite{GK, Dr, DL, BS1, BS2, Bou, dFEI, CNS}.\medskip

\goodbreak 

Our first result is the description of the $\t$-adically local geometry of \arquile varieties (see Thms.~\ref{singularlocus} and \ref{partition} for the precise statements).


\begin{theorem} [\Arqregular and \arqsingular points] \label{regularpoints} Let $\YY=\YY(f)$ be \an \arquile variety in $\AAA^m$, where $\AAA$ denotes one of the rings $\k[[\t]]$, $\k\{\t\}$, or $\k\langle\t\rangle$, and where $f\in\BB^k$ is a power series vector in $\t$ and $\y_1,...,\y_m$ of the same quality as $\AAA$.%
\footnote{ In the case of $\AAA=\k\{\t\}_s$, assertion (a) is taken as the definition of \arquile regularity, see the preceding footnote.}
\vs.2cm

{\rm (a)} The \arquile singular locus $\Sing(\YY)$ of $\YY$ is \an \arquile closed, proper subset of $\YY$. It can be defined in $\YY$ by the vanishing of appropriate minors $g_j(\t,\y)$, $j=1,...,q$, of the relative Jacobian matrix $\6_\y f(\t,\y)$ of $f$, 
$$\Sing(\YY)=\{y(\t)\in\YY,\, g_1(\t,y(\t))=\ldots =g_q(\t,y(\t))=0\}.$$

{\rm (b)} Every \arqregular point $y\in\Reg(\YY)$ of $\YY$ has a $\t$-adically open and textile locally closed neighborhood $\WW$ which is textile isomorphic to a free $\AAA$-module $\AAA^s$,
$$\Phi:\WW\mapname{\isom} \AAA^s.$$
\end{theorem}


\begin{rmks} (a) It follows from assertion (a) that the iterated \arqsingular loci $\,\Sing_{i+1}(\YY)=\Sing(\Sing_i(\YY))$ of $\YY$, with $\Sing_0(\YY)=\YY$, define a {\it finite} filtration of $\YY$ by \arquile closed subvarieties. Their set-theoretic differences admit the analogous local description as $\YY$.\medskip

(b) Assertion (b) says that $\t$-adic locally at an \arqregular point $y$ of $\YY$, \an \arquile variety is ``smooth'' in the sense of differential geometry. The neighborhood $\WW$ and the isomorphim $\Phi$ are constructed explicitly in the proof. In general, it seems that the neighborhood $\WW$ cannot be chosen \arquile open or at least textile open, cf.~\cite{BS3}. As $\Phi$ is a textile isomorphism, it is also a $\t$-adic homeomorphism.
\end{rmks}


Assertion (b) of the preceding theorem is a consequence of the principal result of this article, which describes the \arquile local geometry of \arquile varieties. The \arquile topology is coarser than the $\t$-adic topology appearing in Theorem \ref{regularpoints}, so the statements below are stronger (see Theorem \ref{factorization} for a more detailed statement). Results of a similar flavour have been obtained by Bouthier for arc spaces \cite{Bou}. However, the arguments and the formulation of the results there are somewhat different from ours; it would be interesting to figure out the precise relation between the two approaches.


\begin{theorem} [Structure of \arquile varieties] \label{structuretheorem} Let $\YY=\YY(f)$ be \an \arquile variety in $\AAA^m$, where $\AAA$ denotes one of the rings $\k[[\t]]$, $\k\{\t\}$, $\k\{\t\}_s$, or $\k\langle\t\rangle$, and where $f\in\BB^k$ is a power series vector in $\t$ and $\y_1,...,\y_m$ of the same quality as $\AAA$. Cover the \arquile regular locus $\Reg(\YY)$ of $\YY$ by the \arquile open subsets 
$$\UU_j=\{y(\t)\in\YY,\, g_j(\t,y(\t))\neq 0\},$$
where $g_1(\t,\y),\ldots,g_q(\t,\y)$ denote minors of the relative Jacobian matrix $\6_\y f(\t,\y)$ of $f$ defining $\Sing(\YY)$ as in Theorem \ref{regularpoints}. For $j=1,...,q$ and $d\in\N$, let 
$$\YY_{jd}=\{y(\t)\in\UU_j,\, \ord_\t\, g_j(\t,y(\t))=d\}$$
denote the textile locally closed stratum in $\UU_{j}$ of vectors $y(\t)$ for which $g_{j}(\t,y(\t))$ has order $d$ as a series in $\t$. Then each $\YY_{jd}$ is (naturally) textile isomorphic to a cartesian product of a (finite dimensional) algebraic variety $\ZZ_{jd}$ with a free $\AAA$-module $\AAA^{s}$,
$$\Phi_{jd}: \YY_{jd}\mapname{\isom} \ZZ_{jd}\times \AAA^{s}.$$
\end{theorem}


\begin{rmks} (a) It is necessary for the factorization to cover $\YY$ first by the \arquile open sets $\UU_j$ and to then work with each $\UU_j$ individually. These are countable unions of the strata $\YY_{jd}$, for $d\in\N$, and each $\YY_{jd}$ can be factorized as indicated above. By assertion (a) of Theorem \ref{regularpoints} the covering and factorization of the \arqregular locus of $\YY$ can be repeated for the \arquile locally closed strata $\YY_i=\Sing_i(\YY)\sm\Sing_{i+1}(\YY)$ of the \arqregular points of $\Sing_i(\YY)$.\medskip

(b) The varieties $\ZZ_{jd}$ are in general singular (some of them may be empty). They can be equipped naturally with a (not necessarily reduced) scheme structure. All strata and isomorphisms will be constructed explicitly. The textile isomorphisms $\Phi_{jd}$ are compositions of a division map with \an \arquile map, and are, in particular, $\t$-adic homeomorphisms. The smooth factor $\AAA^{s}$ appears naturally through the construction of the map $\Phi_{jd}$ though it is abstractly textile isomorphic to $\AAA$; the chosen factor will be the same for all $j$ and $d$. \medskip

(c) For convergent $y(\t)$ and $f$, the isomorphisms $\Phi_{jd}$ are analytic in the sense of locally convex spaces, see \cite{HM}. In the case where $\AAA=\k\{\t\}_s$, for $s>0$, the statement of the theorem has to be slightly modified: First, it holds only for radii $s$ which are sufficiently small with respect to the given $f$, secondly, one has to restrict $\YY_{jd}$ to a sufficiently small ball with respect to the $s$-norm. The isomorphisms $\Phi_{jd}$ are then Banach-analytic. This relates to the theory of short arcs, in particular to a question of Koll\'ar-N\'emethi, see conjecture 73 in \cite {KN}, as well as \cite{JK, Bou}.\medskip

(d) It is not hard to see that the theorem implies assertion (b) of Theorem \ref{regularpoints} by applying it to $\YY^\circ:=\Reg(\YY)$, the set of \arqregular points of $\YY$, and taking a sufficiently small $\t$-adic neighborhood $\WW$ of the \arqregular point $y$ such that $\Phi_{jd}$ maps $\WW$ to $\{0\}\times \AAA^{s}$. Note that $\WW$ can be taken $\t$-adic open because of (a) of Theorem \ref{regularpoints}.\end{rmks}


The proof of Theorem \ref{structuretheorem} is based on a geometric interpretation of the proofs of Artin and P\l oski of their respective approximation theorems for the power series solutions of formal and analytic equations \cite{Ar1, Ar2, Pl}. This geometric viewpoint allows us to transfer linearization techniques from differential geometry to the context of \arquile maps between power series spaces: the defining equations of the strata $\YY_{jd}$ will be linearized up to a finite dimensional subspace of $\AAA^m$ by composing them with suitable textile isomorphisms, see Theorem \ref{linearizationintro} below. The linear part of the equations defines the smooth factor $\AAA^{s}$, the non-linear part the variety $\ZZ_{jd}$. The linearization of the defining equations of $\YY_{jd}$ shows that the trivializing isomorphisms $\Phi_{jd}$ of the structure theorem are actually embedded (i.e., stem from isomorphisms of (suitable strata of) the affine ambient space $\AAA^m$) and, moreover, induce a natural scheme structure on the finite dimensional factor $\ZZ_{jd}$. \medskip

As a direct consequence of the structure theorem, one obtains Artin's analytic, algebraic and strong approximation theorems in the univariate case in the following version, cf.~\cite{Ar1, Ar2, Pl, Gre}.
\goodbreak


\begin{corollary} \label{artin} (a) Let $f(\t,\y)$ be a convergent, respectively algebraic power series vector. The sets of convergent, respectively algebraic power series solutions $y(\t)\in\k\{\t\}^m$, respectively $y(\t)\in\k\langle\t\rangle^m$, of $f(\t,\y)=0$ are  $\t$-adically dense in the set of formal power series solutions $\wh y(\t)\in\k[[\t]]^m$.\medskip

(b) Let $f$ be a formal power series vector. For all $c\in\N$, every $\t$-adically sufficiently good approximate solution $\ol y(\t)$ of $f(\t,\y)=0$ admits an exact formal solution $\wh y(\t)$ which coincides with $\ol y(\t)$ up to degree $c$.
\end{corollary}


Indeed, the structure theorem reduces the approximation problem to the case of linear equations, and there the flatness of the formal power series ring over the rings of convergent and algebraic power series implies the existence of the respective power series solutions, see remark \ref{proofartin} for the details of this reasoning. This method also works for power series in several variables and provides a ``geometric" proof for Artin approximation. The technicalities are, however, more involved \cite{Ha2}. \medskip\goodbreak


{\it Formally local geometry.} There is a second, more deformation-theoretic approach to the geometry of \arquile varieties: Instead of decomposing them into strata which are shown to be ``globally'' cartesian products, one can also study the germs or formal neighborhoods of \arquile varieties $\YY$ at selected \arqregular points $y$. This corresponds to work with the local rings $\O_{\YY,y}$ of $\YY$ at $y$ (as defined below) and their completions $\wh\O_{\YY,y}$. The respective result, first established by Grinberg-Kazhdan and Drinfeld in the special case of arc spaces, can be understood as a partial generalization of Cohen's structure theorem to the infinite dimensional setting: Cohen's theorem asserts that a complete local Noetherian ring $S$ containing a  field is isomorphic to a factor ring of a ring of formal power series in finitely many variables over this field, see Thm.~29.7 in \cite{Ma}. And $S$ is regular if and only if it is isomorphic to a ring of formal power series in finitely many variables. In the infinite dimensional context, the statement is more complicated, see Theorem \ref{cohen} below, and only provides a description of the rings $\wh \O_{\YY,y}$ for \arqregular points $y$ of $\YY$. A different formulation in terms of deformations is given in Theorem \ref{deformations}. \medskip

We restrict in the statement below to the case of formal power series. An analogous statement holds for convergent power series, considering $\k\{\t\}$ as the inductive limit of the Banach spaces $\k\{\t\}_s$, for $s>0$, as in \cite{GR, HM}. Vectors $y(\t)=(y_1(\t),...,y_m(\t))\in \AAA^m$ will be expanded componentswise into $y_i(\t)=\sum_{j\geq 1} y_{ij}\cdot \t^j$ with coefficients $y_{ij}\in\kkk^m$. Let as above $\kkk[\y_{m,\infty}]$ denote the polynomial ring in countably many variables $\y_{ij}$, for $1\leq i\leq m$ and $j\geq 1$. If $\YY=\YY(I)$ is \an \arquile variety in $\AAA^m$ defined by the ideal $I$ of $\BB$, the elements of $I$ induce, by comparison of the coefficients of $\t^\ell$, polynomials $F_\ell\in \kkk[\y_{m,\infty}]$, $\ell\in \N$, which define the equations between the coefficients $y_{ij}$ of points $y$ of $\YY$. Denote by $I_\infty$ the resulting ideal of $\kkk[\y_{m,\infty}]$, so that $\kkk[\y_{m,\infty}]/I_\infty$ is the (infinite-dimensional) coordinate ring of $\YY$. Let $\mm_{y,\infty}$ be the maximal ideal of $\kkk[\y_{m,\infty}]$ generated by all $\y_{ij}-y_{ij}$, for $1\leq i\leq m$ and $j\geq 1$, and let $\nn_{y,\infty}$ be its image in $\kkk[\y_{m,\infty}]/I_\infty$. The local ring $\O_{\YY,y}$ of $\YY$ at $y$ is given by the localization 
$$\O_{\YY,y}=(\kkk[\y_{m,\infty}]/I_\infty)_{\nn_{y,\infty}}.$$ 
Let $\wh\O_{\YY,y}=\varprojlim\, \O_{\YY,y}/\nn_{y,\infty}^k$ be the inverse limit of the factor rings $\O_{\YY,y}/\nn_{y,\infty}^k$.%
\footnote{ The ring $\wh\O_{\YY,y}$ will in general not be complete with respect to the $\wh\nn_{\YY,y}$-adic topology, for $\wh\nn_{\YY,y}$ the maximal ideal of $\wh\O_{\YY,y}$. It is only complete with respect to the inverse limit topology, see the section on deformations. As such, it can be identified with the ring $\kkk[[\y_{m,\infty}]]/\wh I_\infty$, where $\wh I_\infty$ denotes the {\smallit closure} of the ideal $I_\infty\cdot \kkk[[\y_{m,\infty}]]$ in the inverse limit topology, see section 1.2.1 in \cite{Ch}.}
We call this ring the {\it completed local ring} of $\YY$ at $y$. It defines the {\it formal neighborhood} $(\wh {\YY,y})$ of $y$ in $\YY$.

\goodbreak

\begin{theorem} [Grinberg-Kazhdan-Drinfeld factorization theorem] \label{cohen} %
Let $\YY=\YY(f)\subset\AAA^m$ be \an \arquile variety defined by a polynomial vector $f(\t,\y)$, and let $y=y(\t)$ be an \arqregular point of $\YY$. Then the completed local ring $\wh \O_{\YY,y}$ is isomorphic to a formal power series ring $C[[\v_\infty]] =C[[\v_1,\v_2,...]]$ in countably many variables with coefficients in a $\kkk$-algebra $C$ which is the completion of a localization of a $\kkk$-algebra of finite type, 
$$\wh\O_{\YY,y}\isom C[[\v_\infty]].$$

\end{theorem} 


\begin{rmk} The ring $\wh\O_{\YY,y}$ can be characterized by deformations of $y$ over test-algebras (local algebras with nilpotent maximal ideals), and this is the way how Grinberg-Kazhdan and Drinfeld prove the respective factorization result \cite{GK, Dr, BH}. Additional information on this result can be found in \cite{BS1, Bou, BNS, CNS, NS}. In the present text, we will establish the factorization more generally for {\it all} deformations of $y$, with parameters varying in an arbitrary complete local ring $S$ (the topology of $S$ need not be defined by the powers of an ideal); see Theorem \ref{deformations} for the precise formulation.%
\footnote{ The completeness assumption on $S$ is crucial for the proof to work. We do not need, however, the nilpotence of the maximal ideal.  The referee has informed us that the assumption of nilpotence is also not required for Drinfeld's proof.}\medskip

In the paper \cite{BS1}, Bourqui and Sebag have shown that at an {\it \arqsingular} point $y$, \an \arquile variety $\YY$ may not admit a factorization of its formal neighborhood as above. They prove this for the origin $y=0$ of the \arquile subvariety $\YY$ of $\AAA^2$ given as the arc space $X_\infty$ of the plane curve $X:\y_1^2+\y_2^2=0$ over a field not containing a square root of $-1$. The constant arc $y(\t)=0$ is entirely contained in the singular locus of $X$ and hence an \arqsingular point of $\YY$, in the sense defined above. We present in the section on triviality an approach of how one could try to prove directly that, for all hypersurfaces $X\subset \A^m_\kkk$ defined by Brieskorn polynomials $\y_1(\t)^{c_1}+\ldots +\y_m(\t)^{c_m}=0$ over a field of zero characteristic with exponents $c_i\geq 2$, the formal neighborhood of $\YY$ at $0$ does not even admit a one-dimensional smooth factor.\medskip

It is tempting to try to extend the Grinberg-Kazhdan-Drinfeld theorem and the observation of Bourqui and Sebag to an equivalence statement in the spirit of Cohen's structure theorem: A constant arc $y(\t)= const$ of an arc space $\YY=X_\infty$ is \arqregular if and only if $\wh \O_{\YY,y}$ equals a formal power series ring in countably many variables over a (possibly singular) completion of the localization of a $\kkk$-algebra of finite type. \medskip


For non-constant arcs one has to be cautious because of the following example of Hickel, cf.~\cite{Hi}, example 4.6. Consider the Whitney-umbrella $X$  in $\A^3_\kkk$ defined by $\y_1^2-\y_2^2\cdot \y_3=0$, and the arc $y(\t)=(0,0,\t)$. It is entirely contained in the singular locus of $X$, which is the $\y_3$-axis. Let $\wt y(\t)=(a(\t), b(\t), \t+c(\t))$ be a deformation of $y(\t)$, with $a(0)=b(0)=0$ and $c(\t)$ of order at least $2$ in $\t$. Substitution in the defining equation gives
$$a^2(\t)-b^2(\t)\cdot (\t+c(\t))=0,$$
which has, by comparison of orders, the only solution $a(\t)=b(\t)=0$, and $c(\t)\in \t^2\cdot\kkk[[\t]]$ arbitrary. Therefore any deformation of $y(\t)$ of this type already lies in the arc space $({\rm Sing}\, X)_\infty$ of the singular locus of $X$. 
\end{rmk}


{\it Outline of the proof of the structure theorem} : We briefly describe how the stratification of the variety $\YY$ and the factorizations of the strata $\YY_{jd}$ in Theorem \ref{structuretheorem} are obtained.%
\footnote{ The term {\smallit stratification} is used loosely here: the textile closures of strata are in general not finite unions of strata, but just countable unions.}
%
%
Instead of proving directly that the strata $\YY_{jd}$ are isomorphic to cartesian products $\ZZ_{jd}\times\AAA^{s}$ we rather work with the defining system of equations $f(\t,\y(\t))=0$. We will show that such a system of equations can be {\it linearized} locally after a suitable stratification of the ambient space $\AAA^m$ by a suitable textile isomorphism  -- but this is only possible up to a {\it finite dimensional} part. From this it will then be relatively easy to deduce that the zeroset $\YY$ has the required cartesian product structure on each stratum. Let us give some more details of how the linearization works.\medskip


Define, for $d\in\N$, and a suitably chosen minor $g$ of the relative Jacobian matrix $\6_\y f=(\6_{\y_i}f_j)$ of $f$ with respect to $\y$, subsets $\SS_d=\SS_d(g)$ of $\AAA^m$ as the sets of vectors $y\in\AAA^m$ for which $g$ has non-zero evaluation $g(y)$ of order $d$ as a power series in $\t$. The sets $\SS_d$ are \ttextile cofinite and \ttextile locally closed in $\AAA^m$ in the sense that the condition $\ord\,g(y)=d$ carries only on the (finitely many) coefficients $y_{ij}$, $1\leq i \leq m$, $j\leq d$, of vectors $y\in\AAA^m$. It is not difficult to see that the order condition is given by (finitely many) polynomial equalities $=$ and inequalities $\neq$ on these coefficients.%
\footnote{ Sets given by such order conditions are also called {\smallit definable} in the literature.}
\medskip

On each stratum $\SS_d$, the vectors $y\in\SS_d$ will be decomposed uniquely by Weierstrass division into $y=\zz+\vv$ where $\zz$ is a polynomial vector of degree $\leq d$ in $\t$ and $\vv$ is a power series vector whose components have all order $>d$ in $\t$. As $y$ varies in the stratum $\SS_d$, the first summand $\zz$ will vary in a finite dimensional space. This type of decomposition is a standard technique in Artin approximation and in various results on arc spaces. It will allow us to show that $\ff(y)=\ff(\zz+\vv)$ can be {\it linearized} with respect to $\vv$ by an isomorphism $\chi_d$ of $\SS_d$. 
The precise  statement is as follows (see Theorem \ref{linearization} for more details).
\goodbreak


\begin{theorem} [Linearization of \arquile maps] \label{linearizationintro} Let be given \an \arquile map $\ff:\AAA^m\map \AA^\kk,\, y(\t)\map f(\t,y(\t))$  induced by the substitution of the $\y$-variables by a vector $y(\t)$ of power series in a given power series vector $f\in\BB^\kk$. Assume that $\kk\leq m$ and let $g$ be a $(\kk\times \kk)$-minor of the relative Jacobian matrix $\6_\y f$ of $f$. Set $\SS_d= \{y\in\AAA^m,\, g(y)\neq 0$,\, $\ord\, g(y)=d\}$. 

There then exists, for each $d\in\N$, a textile isomorphism $\chi_d:\ZZ_d\times \VV_d\mapname{\isom} \SS_d$ over $\ZZ_d$ from the cartesian product of the quasi-affine algebraic variety $\ZZ_d=\{\zz\in \k[\t]_{\leq d}^m,\, \ord\, g(\zz)=d\}$ with the $\AAA$-module $\VV_d=\langle \t\rangle^d\cdot\AAA^m$ such that the composition
$$\ff\circ\chi_d:\ZZ_d\times\VV_d\mapname{\isom} \SS_d\mapname{f_\infty} \AA^\kk$$
is linear in the second component and of the form
$$(\ff\circ\chi_d)(\zz,\vv) = f(\zz)+\6_\y f(\zz)\cdot\vv.$$
\end{theorem}


This statement can be rephrased by saying that $\ff$ is linearized on each stratum $\SS_d$ {\it up to a finite dimensional summand} by an isomorphism of the source space. If the minor $g$ is identically zero on $\AAA^m$ all sets $\SS_d$ are empty and the statement is vacuous. So the interesting case occurs when $g$ is different from zero. The assumption that $g$ stems from a submatrix of $\6_\y f$ of size $\kk$ (and hence the number $\kk$ of components of $f$ is less than or equal to the number $m$ of $\y$-variables) is essential for the linearization to work. A substantial part of the proof of Theorem \ref{structuretheorem} is dedicated to establishing this assumption on the sets of a suitable partition of $\YY(f)$.\medskip

The construction of the linearization of $\ff$ is the main technical ingredient of the paper. It will be given explicitly by describing the isomorphism $\chi_d$ as a composition of a $\kkk$-linear isomorphism given by the Weierstrass division and an invertible \arquile map. Here, the polynomial vector $\zz$ plays the role of a parameter, and $\ff\circ\chi_d$ will depend in \an \arquile way on $\zz$ (i.e., is a power series in $\zz$), whereas it is linear in $\vv$. The isomorphism $\chi_d$ is textile in both $\vv$ and $\zz$. Once the linearization of $\ff$ with respect to $\vv$ is achieved a quite direct  argument establishes the cartesian factorization of the stratum $\YY_d=\YY(f)\cap\SS_d$ of $\YY(f)$ as described in Theorem \ref{structuretheorem}.
\medskip

For arbitrary vectors $f$, the zeroset $\YY(f)$ has first to be decomposed into the \arqregular and \arqsingular locus. On the first, there exists a covering by \arquile open subsets for which linearization works. On the second, one iterates the decomposition into the \arqregular and \arqsingular locus and then argues by Noetherian induction to construct the required partition. Combining all strata of all subsets obtained in this way then yields the required partition $\YY=\bigcup \YY_i$ together with the individual cartesian product factorizations as described in the structure theorem \ref{structuretheorem}.\medskip

\goodbreak
\medskip


{\it Organization of paper} : After a compilation of definitions in section 2, we recall the Weierstrass division theorem in section 3 and adapt the statement to our purposes. This can be skipped on first reading, as well as the subsequent section 4 on division modules. The main constructions appear in section 5, where the linearization theorem for \arquile maps between power series spaces is formulated and proven. This is the central part of the paper. The proof is elementary but a bit tricky. This result is then exploited in sections 6, 7 and 9 to establish the various factorization theorems. The intermediate section 8 constructs the appropriate covering and stratification of \an \arquile variety. In section 10 we illustrate our techniques in an explicit example. The last section 11 discusses the analytic triviality of \arquile varieties, i.e., the appearance of a smooth factor in the formal neighborhood of a point. A list of symbols can be found after the references.\medskip


{\it Acknowledgements}: The initial phase of the work on this article was carried out in a cooperation of the authors with G.~Rond. Numerous discussions with C.~Chiu helped to clarify various ambiguities. We are also indebted to D.~Popescu, B.~Lamel, S.~Perlega, H.~Kawanoue, F.~Castro-J\'\i menez, M.E.~Alonso, M.~Spivakovsky, and K.~Slavov for many helpful conversations and suggestions.\medskip

Special thanks go to the anonymous referee: It is seldom to obtain such a competent and detailed report as it was the case here. The referee understood the various facets of the article in depth; her/his comments and suggestions were thoughtful and  invaluable.\medskip


\noindent 
\section{\Arquile varieties}\label{arquile}


In this section, we collect the basic material on \arquile varietes needed in later sections. To easen the orientation of the reader, some overlap with section \ref{section_concepts} was admitted.\medskip

Our ground field $\kkk$ will be a perfect field of arbitrary characteristic. It will be assumed to be equipped with a valuation $\abs -:\kkk\map \R$ whenever we talk about convergent power series. To easen the notation we often suppress the dependence of power series $y=y(\t)$ on the variable $\t$. Vectors $f=f(\t,\y)$ will be written as columns of length $k$, vectors of variables $\y$ and $\z$ and of power series $y=y(\t)$ as rows of length $m$. Accordingly, the relative Jacobian matrix $\6_\y f$ of $f$ will be a $(\kk\times m)$-matrix.\medskip


{\it Spaces of power series}: 
We denote by $\AAA$ the ring of formal, respectively convergent or algebraic power series $y(\t)$ in a single variable $\t$ over $\kkk$ and without constant term, $y(0)=0$.%
\footnote{ All results hold, with the appropriate modifications, also for the ring $\k\{\t\}_s$ of power series  with finite $s$-norm. This case will not be mentioned explicitly in the sequel, since it requires slightly different statements (e.g., to restrict the strata of \arquile varieties $\YY$ to neighborhoods of $y$ in the Banach space topology).}
This is a commutative ring without one element $1$. For $m\in \N$ we will consider vectors $y=(y_1,...,y_m)\in\AAA^m$ of power series without constant terms and call accordingly $y_{ij}=(y_{ij},\, 1\leq i \leq m,\, j\geq 1)\in(\kkk^{\N_{>0}})^m$ their coefficient vector. We will interpret $\AAA^m$ as an infinite dimensional affine space over $\kkk$. In the setting of formal power series, its ``coordinate ring'' is the polynomial ring $\kkk[y_{m,\infty}]=\kkk[\y_{ij},\, 1\leq i \leq m,\, j\geq 1]$ in countably many variables. %
\medskip



{\it \Arquile maps}: 
For variables $\y=(\y_1,...,\y_m)$, let $\BB$ denote one of the rings $\kkk[[\t,\y]]$, $\kkk\{\t,\y\}$ or $\kkk\langle\t,\y\rangle$ of power series in $\t$ and $\y$. If $f$ is in $\BB$ and $y=y(\t)\in\AAA^m$, we may replace in $f$ the $\y$-variables by the vector $y(t)$, which we abbreviate by $f(y)=f(\t,y(\t))$. Here and in the sequel, the qualities of $\AAA$ and $\BB$ (formal, convergent, or algebraic) are tacitly accorded. Every vector $f\in\BB^\kk$ induces \an \arquile map
$$\begin{array}{rcl} \ff: \AAA^m& \langlew & \AA^k,\\
y & \langlew & f(y)\end{array}$$
by evaluation at $y$. If $f$ is a linear polynomial in $y_1,...,y_m$ with coefficients in $\AAA$, the induced map $\ff$ is $\AAA$-linear. If $\SS$ is a subset of $\AAA^m$ we may restrict $\ff$ to $\SS$ to get \an \arquile map $\ff :\SS\map \AA^\kk$. \medskip


\begin{theorem}[Inverse function theorem] \label{inv} For any vector $p(\t,\y)\in\BB^m$ whose components are all of order at least $2$ in $\t,\y$, the \arquile map
$$\begin{array}{rcl} \phi: \AAA^m& \langlew & \AAA^m,\\
y & \langlew & y+p(y),\end{array}$$
is \an \arquile isomorphism.  
\end{theorem}


\begin{rmk} In contrast to the finite dimensional setting, and due to the fact that we restrict to power series without constant term, the map $\phi$ is a {\it global} \arquile isomorphism.
\end{rmk}

\begin{proof} Set $u= \y+p(\t,\y)\in\BB^m$ so that $\phi=\uu$. We have $u(0,0)=0$ and the relative Jacobian matrix $\6_{\y}u(0,0)$ is invertible by the assumption on $p$. Let $v\in\BB^m$ be the unique vector with $v(0,0)=0$ and $u(\x,v(\x,\y))=\y$ as given by the finite dimensional implicit function theorem in the formal, convergent and algebraic context (in the last case, it is equivalent to Hensel's lemma, see \cite{Nag}). Then $v$ defines a globally defined \arquile map $\psi=v_\infty:\AAA^m \map  \AAA^m$ via $y\map v(y)$. Clearly $\psi\circ\phi$ is the identity map on $\AAA^m$.\end{proof} 


{\it \Arquile varieties}: Every ideal $I$ of $\BB$ defines a subset $\YY$ of $\AAA^m$ by
$$\YY=\YY(I)=\{y\in\AAA^m,\, f(y)=0 \text{ for all } f\in I\}.$$
We call such sets {\it \arquile subvarieties} of $\AAA^m$, or simply {\it \arquile  varieties}. We also say that these sets are {\it \arquile closed} in $\AAA^m$. They are the fibers $f_\infty^{-1}(0)$ of \arquile maps $f_\infty:\AAA^m\map\AAA^\kk$ induced by a vector $f\in\BB^\kk$ with $f(0)=0$, where the components of $f$ generate the ideal $I$ in $\BB$. Note that \arquile open subsets of $\AAA^m$ are $\t$-adically open.\medskip

Given two ideals $I$ and $J$ of $\BB$ we have $\YY(I\cap J)=\YY(I)\cup\YY(J)$ and $\YY(I+ J)=\YY(I)\cap\YY(J)$. It is clear that if $\sqrt I$ denotes the radical of $I$, then $\YY(I)=\YY(\sqrt I)$. So we may assume without loss of generality that \arquile subvarieties $\YY(I)$ are defined by radical ideals $I=\sqrt I$.%
\footnote{ In this article, we defer from considering $\YY$ as a scheme. All our \arquile varieties will just be considered as subsets of $\AAA^m$, together with their equations, and are thus ``affine'' varieties. There will be no need to glue them or to define them as abstract varieties.}
For $\YY$ an arbitrary subset of $\AAA^m$ we define 
$$I_{\YY}=\{f\in\BB,\, f(y)=0 \text{ for all } y\in \YY\}$$
as the ideal of $\BB$ of power series vanishing upon the substitution of the variables $\y$ by elements $y$ of $\YY$. The ideal $I$ is called {\it saturated} if $I$ equals its {\it saturation} $I_{\YY(I)}$. The saturation of an arbitrary ideal $I$ is always radical; for the converse in the convergent case, see Theorem \ref{saturation} below.


\begin{prop} \label{primedecomposition} Let $I$ be a saturated proper ideal of $\BB$, with irredundant prime decomposition $I=I_1\cap \ldots \cap I_s$. Then all $I_i$ are saturated,  say $\YY(I)=\YY(I_1)\cup\ldots\cup \YY(I_s)$ with $I_i=I_{\YY(I_i)}$, and $\YY(I_i)\neq \emptyset$ for all $i$. 
\end{prop}


The \arquile varieties $\YY(I_i)$ will be called the {\it irreducible components} of $\YY$ (they are the irreducible components with respect to the arquile topology). While the prime components of a saturated ideal are again saturated, the saturation of a prime ideal need not be prime, cf. the example after the proof.


\begin{proof} Suppose that $I_1$ is not saturated and hence properly contained in $I_{\YY(I_1)}$, and let $I_{1,1},\dots, I_{1,r}$ be an irredundant prime decomposition of $I_{\YY(I_1)} = \bigcap_{j=1}^{r} I_{1,j}$. Then $I \subset \bigcap_{j=1}^{r} I_{1,j} \cap I_{2} \cap \dots \cap I_{s}$. Since the subvarieties of $\AAA^m$ associated to both ideals coincide and $I$ is saturated, we obtain that $I = \bigcap_{j=1}^{r} I_{1,j} \cap I_{2} \cap \dots \cap I_{s}$. This prime decomposition can be reduced to an irredundant one, and then $I_{1}$ has to appear in the decomposition. The inequality $I_{i} \neq I_{j}$ for $i\neq j$ implies that $I_{1} = I_{1,l}$ for some $l$, which leads to $I_{\YY(I_1)} \subset I_{1}$, contradictory to $I_1\subsetneq I_{\YY(I_1)}$. Hence all associated primes of $I$ are saturated. Finally, if $\YY(I_{i})$ would be empty, then $I_{i} = I_{\YY(I_{i})} = \BB$, which is not possible since $I_{i}$ is prime in $\BB$ and hence properly contained. 
\end{proof}


\begin{example}\label{examplereducible} A theorem of Kolchin asserts that if $X\subset\A^m_\kkk$ is an irreducible algebraic variety over a field $\kkk$ of zero characteristic, the (global) arc space $X_{\infty}$ consisting of arcs $y(\t)$ on $X$ with arbitrary starting point $y(0) \in X$ is irreducible in the textile topology on ${\mathcal A}^m=\kkk[[\t]]^m$, see \cite{Ko}, Chap.~IV, Prop.~10, or \cite{IK}, Lemma 2.12.%
\footnote{ The set of centered arcs $X_{\infty,0} = \{y  \in X_{\infty},\, \ y(0) = 0\}$ may nevertheless split into several components.}
\medskip

%

In the case of convergent power series over $\kkk=\C$ and if the ideal $I$ is generated by series in $\C\{\y\}$ not depending on $\t$, being radical implies that $I$ is saturated, and $\YY(I)$ is irreducible if and only if $I$ is prime, see Cor.~\ref{radicalone}. And if one starts with a polynomial ideal $I \subset \C[\y]$, then $\YY(I)$ is irreducible if and only if $I\cdot\C\{y\}$ is prime, see Cor.~\ref{radicaltwo}.
\end{example}


{\it Textile maps}: %
A map $\tau: U\subset \AAA^m\map \AA^\kk$ is called {\it textile} if each coefficient of the power series expansion of the images $\zz(\t)=\tau(y(\t))$ of vectors $y(\t)=(y_1(\t),...,y_m(\t))\in\AAA^m$ depends polynomially on finitely many of the coefficients of the expansions of the vectors $y$: There exist, for $\ell\geq 0$, polynomial vectors $F_{\ell}=(F_{1\ell},...,F_{\kk\ell})\in\kkk[\y_{m,\infty}]^\kk$ in countably many variables $\y_{m,\infty}=\{\y_{ij},\, 1\leq i\leq m,\,j\geq 1\}$, each of them only depending on finitely many of the variables, such that, writing $y_j(\t)=\sum_{i\geq 0} y_{ij}\cdot \t^i$, we have
$$\zz(\t)=\tau(y(\t))=\sum_{\ell\geq 0} F_{\ell}(y_{m,\infty})\cdot \t^\ell.$$
Textile maps induce ring homomorphisms
$$\alpha_\tau: \kkk[\z_{\kk,\infty}]\map \kkk[\y_{m,\infty}]$$
defined by $\alpha_\tau(\z_{n\ell})=F_{n\ell}(\y_{m,\infty})$, for $n=1,...,\kk$, and $\ell\geq 0$. Conversely, any $\kkk$-algebra homomorphism $\alpha :\kkk[\z_{\kk,\infty}]\map \kkk[\y_{m,\infty}]$ induces, in the case of formal power series $\AAA=\k[[\t]]$, a textile map $\tau_\alpha:\AAA^m\map \AA^\kk$.%
\footnote{ The definition of textile maps also applies to maps between rings of power series in $\t$ with coefficients in an arbitrary ring $S$ instead of the field $\kkk$. This situation will occur in the section on deformations and the proof of Theorem \ref{deformations}.}
\medskip

Every \arquile map $\ff: \AAA^m\map \AA^\kk,\, y(\t)\map f(\t,y(\t))$, associated to a vector $f\in\BB^\kk$ is textile. Also, if $y(\t)=a(\t)\cdot h(\t)+\zz(\t)$ denotes the componentswise division of $y\in\AAA^m$ by a power series $h(\t)$ of order $d$, the associated {\it division maps} $y\map v=a\cdot h(\t)$ and $y\map \zz$ are textile, see the section on Weierstrass division. These two types of maps together with their compositions and inversions will be the main instances of textile maps in this article. A {\it textile isomorphism} is a bijective textile map whose inverse is again textile. See [BH] for more information on textile maps.
\medskip

A map $\tau: \UU\subset \AAA^m\map \AA^\kk$ defined on a subset $\UU$ of $\AAA^m$ is called {\it rationally textile} if each coefficient of the power series expansion of the images $\tau(y)$ of vectors $y\in\AAA^m$ is a rational function of finitely many of the coefficients of the expansions of the vectors $y$. The denominators of the rational functions are not allowed to vanish when evaluating them on the coefficients of vectors in $\UU$. So a rationally textile map is defined on whole $\UU$. If $y(\t)=a(\t)\cdot h(\t) +\zz(\t)$ denotes the componentswise Weierstrass division of $y\in\AAA^m$ by a series $h(\t)$ of order $d$, the map $y\map a$ is rationally textile, cf. Theorem \ref{weierstrass}.  \medskip


{\it Topologies}:
The affine space $\AAA^m$ over $\AAA$ will be equipped with three topologies. The {\it $\t$-adic topology} defined by the ideals $\langle \t\rangle^d=\t^d\cdot \AAA$ of $\AAA$, for $d\geq 0$,  as a basis of neighborhoods of $0\in\AAA^m$, the {\it \ttextile topology}, whose closed sets are the zerosets of (possibly infinitely many) polynomial equations in the coefficient vectors of the expansions of vectors $y\in\AAA^m$, and the {\it \arquile topology}, whose closed sets are the zerosets of \arquile maps. All three topologies will appear in the sequel. The \arquile topology is Noetherian since a descending chain of \arquile varieties $\YY_1\supset\YY_2\supset\ldots$ corresponds to an increasing chain of saturated ideals $I_1\subset I_2\subset \ldots$ of $\BB$ with $\YY_j=\YY(I_j)$.\medskip

A subset of $\AAA^m$ is {\it \ttextile locally closed} if it is the set-theoretic difference of two \ttextile closed sets. \Arquile subvarieties $\YY(f)$ of $\AAA^m$ as defined above are \ttextile closed and $\t$-adically closed, since \arquile maps $\ff$ are \ttextile continuous and $\t$-adically continuous. Textile maps are textile continuous by definition. They are also $\t$-adically continuous, since each coefficient of a power series vector in the image depends only on finitely many coefficients of the power series vector in the source. In particular, textile isomorphisms are $\t$-adic homeomorphisms.\medskip

\Arquile maps $\ff:\AAA^m\map\AA^\kk$ associated to vectors $f\in\BB^\kk$ map the $\t$-adic neighborhood $y+\langle \t\rangle^c\cdot\AAA^m$ of a vector $y\in\AAA^m$ into the neighborhood $\ff(y)+\langle \t\rangle^c\cdot\AA^\kk$. They are hence uniformly Lipschitz continuous with Lipschitz constant 1 with respect to the $\t$-adic metric on $\AAA$. For division maps $y\map \vv=a\cdot h(\t)$  and $y\map\zz$ sending $y\in \AAA^m$ to the summands $y=\vv+\zz$ of the componentswise division $y=a\cdot h(\t)+ \zz$ of $y$ by a series $h(\t)$ of order $d$, the neighborhood $\langle \t\rangle^c\cdot\AAA^m$ is mapped into $\langle \t\rangle^c\cdot\AAA^m$. In contrast, the map $y\map a$ sending $y$ to the quotient $a$ maps $\langle \t\rangle^c\cdot\AAA^m$ only into $\langle \t\rangle^{c-d}\cdot\AAA^m$, for all $c\geq d$.\medskip

A \ttextile closed subset $\ZZ$ of a subset $\UU$ of $\AAA^m$ is {\it cofinite} in $\UU$ if it can be defined (as a subset of $\UU$) by {\it finitely} many polynomial equations in (finitely many) coefficients of the power series expansions of vectors of $\AAA^m$. In the literature, such subsets are also known as {\it cylindrical} or {\it generically stable}. The analogous definition applies to \ttextile locally closed subsets.\medskip


{\it Singular locus} : Let $y=y(\t)\in\AAA^m$ be a given power series vector. We may associate to $y$ the evaluation map $\varepsilon_y:\BB\map\AAA$ sending the variable $\y$ to $y$. Its kernel $\Ker(\varepsilon_y)$ is a prime ideal of $\BB$ which will be denoted by $\PP_y$,
$$\PP_y=\{f\in\BB,\, f(\t,y(\t))=0\}.$$
It is called the {\it ideal of relations} between the components of $y$. This is, in contrast to the finite dimensional situation of points in affine space $\A^m_\kkk$ over a field, never a maximal ideal. The map sending $y$ to $\PP_y$ is injective, since $y=y(\t)$ can be reconstructed from the elements $\y_i-y_i(\t)$ of $\PP_y$. 
%
\footnote{ As $\BB$ is a ring of power series, one may want to prefer to associate to it the formal spectrum ${\rm Spf}(\BB)$. This is, however, irrelevant for our purposes since we will always work with the rings, the spectra only serving as an intuitive geometric visualization of the algebraic facts.} 
%
We denote by $\BB_{\PP_y}$ the localization of $\BB$ at $\PP_y$.\medskip




Denote by $\ol \PP_y$ the image of $\PP_y$ in $\BB/I$. We then say that $y$ is an {\it \arqregular} point of $\YY$ if the localization $(\BB/I)_{\ol \PP_y}=\BB_{\PP_y}/I\cdot \BB_{\PP_y}$ is a regular local ring. By the above, this is equivalent to saying that $p_y$ is a regular point of the scheme $Y$. The locus of \arqregular points of $\YY$ is denoted by $\Reg(\YY)$, and we set $\Sing(\YY)=\YY\sm \Reg(\YY)$, the {\it \arqsingular locus} of $\YY$. If $y$ is an \arqregular point of $\YY$, it necessarily lies in a single irreducible component of $\YY$, cf.~Proposition \ref{primedecomposition} and the theorem below. For local investigations at \arqregular points we may therefore assume that $I$ is prime. The following version of the Jacobian criterion will be used repeatedly, cf.~\cite{dJP}, Thm.~4.3.15, \cite{Ru}, Prop.~4.3, and \cite{Za}, Thm.~7 and its corollary (recall that the ground field $\kkk$ is assumed to be perfect):


\begin{prop} [Jacobian criterion] \label{jacobian} Let $I$ be an ideal of $\BB$ generated by elements $f_1,...,f_\kk$, and let $\PP$ be a prime ideal of $\BB$ containing $I$. Then $(\BB/I)_{\ol \PP}$ is a regular local ring if and only if the rank of the Jacobian matrix $\6_{\t\y} f$ of $f=(f_1,...,f_\kk)$ modulo $\PP$ equals the height of the ideal $I_\PP=I\cdot \BB_{\PP}$.
\end{prop}


As a corollary, we get

\begin{theorem} [\Arqregular and \arqsingular locus] \label{singularlocus} Let $\YY=\YY(I)$ be \an \arquile variety in $\AAA^m$ defined by a saturated ideal $I$ of $\BB$, and let $y=y(\t)\in\YY$ be a point of $\YY$ with prime ideal $\PP_y$. Denote by $r$ the height of the ideal $I\cdot\BB_{\PP_y}$.\medskip

{\rm (a)} The point $y$ is an \arqregular point of $\YY$, if and only if there exist elements $f_1,...,f_r$ in $I$ and an $(r\times r)$-minor $g$ of the relative Jacobian matrix $\6_{\y}f$ of $f=(f_1,...,f_r)$ with respect to $\y$ which does not lie in $\PP_y$, i.e., which satisfies $g(y)\neq 0$. 

{\rm (b)} For every \arqregular point $y$ of $\YY$, one can choose $f_1,...,f_r\in I$ and the minor $g\in\BB\sm\PP_y$ as in assertion (a) such that $y\in\YY(f)\sm \YY(g)$ and  
$$\YY\sm \YY(g) =\YY(f)\sm \YY(g).$$

{\rm (c)} The \arqregular locus $\Reg(\YY)$ of $\YY$ is covered by finitely many  \arquile open subsets of the form $\YY(f)\sm \YY(g)$ as in assertion (b), for varying values $r$.

{\rm (d)} The intersections $\YY_i\cap\YY_j$ of irreducible components of $\YY$ lie inside $\Sing(\YY)$.

{\rm (e)} The \arqsingular locus  $\Sing(\YY)$ is \an \arquile closed and proper subset of $\YY$. The height of its defining ideal $J$ in $\BB$ is larger than the height of $I$.  


\end{theorem} 


\begin{rmks} (a) Assertion (e) implies the finiteness of the filtration by the iterated \arqsingular loci as claimed in assertion (a) of the structure theorem \ref{structuretheorem}.\medskip

(b) If $X$ is an affine algebraic variety over $\kkk$ with arc space $X_{\infty,0}$ of arcs centered at $0$, the theorem shows that the \arqregular points of $X_{\infty,0}$ in the sense of \arquile varieties correspond to arcs in $X$ which are not entirely contained in its singular locus ${\rm Sing}\, X$.\medskip

(c) The elements $f_1,...,f_r$ of assertion (b) form a regular sequence in $\BB$, i.e., generate an ideal of height $r$. This follows from the Jacobian criterion.\medskip

(d) It follows from assertion (a) that the \arqsingular locus $\Sing(\YY)$ of an irreducible \arquile variety $\YY$ is defined inside $\YY$ by the (saturation of the) ideal generated by all $(r\times r)$-minors of the relative Jacobian matrix $\6_\y f$ with respect to $\y$ of vectors $f=(f_1,...,f_r)$ of elements of $I$.\medskip 

\end{rmks}


\begin{proof} Let us show (a). By the Jacobian criterion, there exist elements $f_1,...,f_r$ in $I$ such that the entire Jacobian matrix $\6_{\t\y}f$ of $f=(f_1,...,f_r)$ has rank $r$ modulo $\PP_y$.  So there is an $(r\times r)$-minor $g$ of $\6_{\t\y}f$ which satisfies $g(y)\neq 0$. Deriving the equation $f(\t,y(\t))=0$ with respect to $\t$ gives $\6_\t f(y(\t))=-\6_\y f(y(\t))\cdot\6_\t y(\t)$. This shows that $g$ can actually be chosen from the relative Jacobian matrix $\6_{\y}f$ of $f$ with respect to $\y$. This proves the ``only if'' part. The converse implication follows directly from the Jacobian criterion.\medskip

To prove (b), the strategy is as follows: select the irreducible component $\YY_1$ of $\YY$ passing through $y$ (anticipating here assertion (d)), and represent it as an irreducible component of a complete intersection defined by some elements $f_1,...,f_r$ in $\BB$. From there, new elements $f'_1,...,f'_r$ will be constructed multiplying $f_1,...,f_r$ with the square $h^2$ of an element $h$ of $\BB$ whose zero-locus $\YY(h)$ contains all components of $\YY$ but $\YY_1$. The chosen minor $g'$ of the Jacobian matrix of $f'=(f'_1,...,f'_r)$ will then be a multiple of $h$ from which we infer that $f'$ and $g'$ have the required properties.\medskip

Denote by $I_0\subset I$ the ideal of $\BB$ generated by elements $f_1,...,f_r$ of $I$ as indicated in assertion (a). The height of $I_0\cdot \BB_{\PP_y}$ is at most $r$. By the choice of $f_1,...f_r$ and the Jacobian criterion it is exactly $r$. This implies that $y$ is an \arqregular point of $\YY(I_0)$, i.e., that the ring $(\BB/I_0)_{\ol \PP_y}=\BB_{\PP_y}/I_0\cdot \BB_{\PP_y}$ is regular. In particular,  $I_0\cdot \BB_{\PP_y}$ is a prime ideal in $\BB_{\PP_y}$. It is the localization $I_1\cdot\BB_{\PP_y}$ of a prime ideal $I_1$ of $\BB$, the ideal defining the irreducible component $\YY_1$ of $\YY$ passing through $y$. Therefore $I_1$ is an associated prime of both $I$ and $I_0$. We can write $I_1=(I_0:h)$ as a colon ideal with some element $h\in\BB\sm\PP_y$. \medskip


We now set $f_i'=h^2\cdot f_i$, for $i=1,...,r$, and let $g'$ denote the minor of $\6_\y f'$ defined by the same columns as those which were taken for $g$. Observe that $h$ divides $g'$. As $g(y)\neq 0$ and $h(y)\neq 0$ we also have $g'(y)\neq 0$. Clearly, $\YY(h)\subset\YY(g')$. Let $I'=h^2\cdot I_0$ denote the ideal of $\BB$ generated by $f_1',...,f_r'$. As $I'\subset I$ we have $\YY(I)\subset\YY(I')$. Conversely, $I_1=(I_0:h)$ implies that $h\cdot I_1\subset I_0$ and then $h^3\cdot I_1\subset I'$. We get
$$\YY(I')\subset\YY(h^3\cdot I_1)=\YY(h)\cup\YY(I_1)\subset \YY(h)\cup\YY(I)$$
and thus
$$\YY(I')\sm\YY(h)\subset \YY(I)\sm \YY(h),$$
from which, by $I'\subset I$, follows equality,
$$\YY(I)\sm \YY(h)=\YY(I')\sm\YY(h).$$
The inclusion $\YY(h)\subset \YY(g')$ finally implies
$$\YY(I')\sm\YY(g')=\YY(I)\sm \YY(g').$$
as was claimed in (b). Assertion (c) is obvious since the arquile topology is Noetherian. \medskip

To prove (d), let $I=I_1\cap\ldots \cap I_s$ be the irredundant prime decomposition of $I$. Order the ideals $I_i$ such that $y$ is contained in $\YY(I_i)$ for $i\leq t$ and not contained in $\YY(I_i)$ for $i>t$, where $1\leq t\leq s$. Then $I\cdot \BB_{\PP_y}= I_1\cdot \BB_{\PP_y}\cap \ldots\cap I_t\cdot \BB_{\PP_y}$. If $y$ is contained in two irreducible components of $\YY$, then $I\cdot \BB_{\PP_y}$ is the intersection of at least two prime ideals and hence itself not a prime ideal. Then $\BB_{\PP_y}/I\cdot \BB_{\PP_y}$ is not regular so that $y\in\Sing(\YY)$.\medskip

We conclude with the proof of assertion (e). The \arqregular locus $\Reg(\YY)$ of $\YY$ is open by (c). For the remaining assertions, assume first that $\YY$ is irreducible, say, that $I$ is prime. Then $(\BB/I)_I$ is a field, in particular regular. By the Jacobian criterion there exist elements $f_1,...,f_r$ in $I$, where $r$ is the height of $I\cdot \BB_I$, and a minor $g$ of $\6_{\t\y}f$ of $f=(f_1,...,f_r)$ which is not contained in $I$. But $I$ is saturated and $\YY=\YY(I)$ is non-empty, so there is a point $y\in \YY=\YY(I)$ for which $g(y)\neq 0$. This implies by the Jacobian criterion that $(\BB/I)_{\ol \PP_y}$ is regular, so that $y$ is an \arqregular point of $\YY$. 
%
%
Hence $\Reg(\YY)$ is non-empty and $\Sing(\YY)$ a proper subset of $\YY$. Using the argument from the proof of assertion (a), we may choose for $g$ a minor of the relative Jacobian matrix $\6_\y f$ of $f$ with respect to $\y$. Let $J$ denote the ideal of $\BB$ defining $\Sing(\YY)$. By definition, $I\subset J$. But $g\in J$ and $g\not\in I$, so that $I$ is strictly contained in $J$. As $I$ is prime, Krull's principal ideal theorem 
ensures that the height of $J$ is larger than the height of $I$. This establishes assertion (e) for prime ideals.\medskip

Let now $I$ be arbitrary. It is easy to see that $\Sing(\YY)$ is composed of the union of the \arqsingular loci $\Sing(\YY_i)$ of the irreducible components $\YY_i$ of $\YY$ together with the union of the pairwise intersections $\YY_i\cap \YY_j$ of these components, for $i\neq j$. So it is defined by the saturation  $J$ of the intersection of the ideals $(I_i+I_j)\cap J_i$ over all $i\neq j$, where $J_i$ denotes the ideal of $\BB$ defining $\Sing(\YY_i)$. We already know by the preceding paragraph that the height of $J_i$ is larger than the height of $I_i$. And as $I=I_1\cap \ldots \cap I_s$ is irredundant, also the height of $I_i+ I_j$ is larger than the heights of both $I_i$ and $I_j$, for all $i\neq j$. It follows that the height of $J$ is larger than the height of $I$. This concludes the proof of the theorem.
\end{proof}


It is possible to characterize saturated prime ideals in terms of \arqregularity, at least for convergent power series over $\kkk = \C$, using Gabrielov's theorem, \cite{Ga, Iz}.


\begin{theorem}\label{saturation} A prime ideal $I \subset \C\{\t,\y\}$ is saturated if and only if there exists a point $y \in \YY(I)$ such that $(\C\{\t,\y\}/I)_{\ol \PP_y}$ is regular.
\end{theorem}


\begin{proof} Assume at first that $I$ is saturated, $I=I_{\YY(I)}$. The \arqregular locus of an \arquile variety is non-empty by assertion (e) of Theorem \ref{singularlocus}, so there exists a point $y\in\YY(I)$ for which $(\BB/I)_{\ol \PP_y}$ is regular. This shows the first implication. Conversely, let $I$ be an arbitrary ideal and assume that $(\BB/I)_{\ol \PP_y}$ is regular for some $y \in \YY(I)$. Let $r$ be the height $\height(I)$ of $I$. By \cite{Pl} or Theorem \ref{factorization} below, there exists a vector of power series $\wt y(\t,\s) \in \C[[\t,\s]]$, where $\s = (\s_{1},\dots, \s_{m-r})$ is a set of new variables, such that $\wt y(\t,s(\t)) \in \YY(I)$ for every $s(\t) \in \AAA^{m-r}$ and with the property that the rank of the relative Jacobian matrix $\6_{\s}(\wt y(\t,\s))$ equals $m-r$. \medskip

Let $\varepsilon: \BB \map \C\{\t,\s\}$, $f(\t,\y) \map f(\t,\wt y(\t,\s))$ be the evaluation map. By Gabrielov's theorem, the rank of $\6_{\t,\s}(\t,\wt y(\t,\s))$ is at most $\dim_\C(\BB/\ker \varepsilon)$. It equals $1 + m-r$. We obtain the estimate $\height(\ker \varepsilon) \leq r$. Now let $f \in I_{\YY(I)}$. Then $f(\t,\wt y(\t,s(\t))) = 0$ for all $\s(\t)\in \AAA^{m-r}$. This shows $f(\t,\wt y(\t,\s)) = 0$, and hence $I\subset I_{\YY(I)} \subset \ker \varepsilon$. Both $\ker \varepsilon$ and $I$ are prime, and as $\height(\ker \varepsilon) \leq \height(I)$, both ideals coincide; in particular $I = I_ {\YY(I)}$.
\end{proof}


\begin{corollary}\label{radicalone}
Let $I \subset \C\{\t,\y\}$ be an ideal generated by convergent power series which do not depend on $\t$.

{\rm (a)} If $I$ is a radical ideal, then $I$ is saturated.

{\rm (b)} If $I$ is prime, then $\YY(I)$ is irreducible.
\end{corollary}\goodbreak


\begin{proof}  Set $J=I\cap\C\{\y\}$ and let $J = J_{1}\, \cap \dots \cap\, J_s$ be an irreducible prime decomposition. Denote by $X_i$ the irreducible analytic space germ defined by $J_i$. Because of the curve selection lemma  \cite{Mi}, we can choose for every component $X_i$ an arc $y^i(\t) \in \AAA^m$ which does not lie completely in the singular locus of $X_i$. Hence $(\BB/J_i\cdot \BB)_{\ol P_{y^i}}$ is regular, which, by the previous theorem, implies that $J_i\cdot  \BB$ is satu\-rated. It is easy to see that $I = J\cdot \BB = (J_1\cdot \BB)\, \cap \dots \cap\, (J_s\cdot  \BB)$. The intersection of finitely many saturated ideals is again saturated, hence $I$ is saturated.\end{proof}


\begin{corollary}\label{radicaltwo}
Let $J \subset \C[\y]$ be a polynomial ideal, and set $I=J\cdot\C\{\t,\y\}$.

{\rm (a)}  If $J$ is a radical ideal, then $I$ is saturated.

{\rm (b)} Let $J$ be radical. Then $\YY(I)$ is irreducible if and only if $J\cdot \C\{\y\}$ is prime.

\end{corollary}


\begin{proof}
For the first statement, it suffices to check, by the preceding corollary, that $J\cdot \BB$ is a radical ideal. But this is ensured by Chevalley's theorem, see \cite{Ru}, Prop.~1.3. The second statement follows from the first together with the previous theorem: If $J$ is prime, then $I$ is prime and the preceding corollary tells us that $\YY(I)$ is irreducible. Conversely, assume that $\YY(I)$ is irreducible. As $J$ is radical, so is $I$, and hence $I$ is saturated by (a). Then, as $\YY(I)$ is irreducible,  $I$ and hence $J\cdot \C\{\y\}$ must be prime. \end{proof}
\goodbreak


\section{Weierstrass division}


We will give two versions of this theorem, one with coefficients in a field and one with coefficients in a complete local ring, both in the case of series in just one variable $\t$. The main point is to describe how the quotient and the remainder of the division depend on the divisor and the series to be divided. In the first case, this requires to work on the strata where the divisor has a fixed order as a power series in $\t$, in the second case one considers deformations of a given divisor obtained by varying its coefficients by elements in the complete local ring. All results in this section are classical and hold equally for formal, convergent and algebraic power series.%
\footnote{ The division of algebraic power series is due to Lafon \cite{La}.}
\medskip


(1) {\it Division on strata} : The results apply to any of the rings $\k[[\t]]$, $\k\{\t\}$, $\k\{\t\}_s$, or $\k\langle\t\rangle$ of power series without constant term.%
\footnote{ In the case $\AAA=\k\{\t\}_s$, one has to restrict to sufficiently small radii $s>0$, and vectors $y$ of sufficiently small $s$-norm, see \cite{HM} for details.}
Denote by $\k[\t]_{\leq d}$ the vector space of polynomials in $\t$ of degree $\leq d$  without constant term. \medskip


\begin{theorem} [Weierstrass division on strata of constant order] \label{weierstrass} Assume given a power series $g\in\BB$ in variables $\t$ and $\y=(\y_1,...,\y_m)$.\medskip


{\rm (a)} Let $y=y(\t)\in\AAA^m$ be a power series vector without constant term such that $g(y)=g(\t,y(\t))$ is non-zero of order $d$. For every $\ww\in\AAA$ there exist a unique power series $a\in\AAA$ and a unique polynomial $\zz\in \k[\t]_{\leq d}$ of degree $\leq d$ such that
$$\ww=g(y)\cdot a+\zz.$$
%

{\rm (b)} Denote by $\SS_d \subset\AAA^m$ the \ttextile locally closed stratum of vectors $y\in\AAA^m$ for which $g(y)$ is non-zero of order $d$. The maps
$$\psi:\SS_d\times \AAA\map \k[\t]_{\leq d}\times \langle\t\rangle^d\cdot\AAA $$
$$(y,\ww)\map (\zz,\vv)$$
and
$$\chi:\SS_d\times \AAA\map \k[\t]_{\leq d}\times \AAA$$
$$(y,\ww)\map (\zz, a)$$
given by the division $\ww=\vv+\zz=g(y)\cdot a+\zz$ are textile, respectively rationally textile.


\end{theorem}




Note that letting $g$ in assertion (a) be independent of $\y$, say, taking $g\in {\mathcal A}$ a series in $\t$ of order $d$, one obtains the usual version of the division. The requirement that the quotient $a\in\AAA$ has no constant term is only imposed for technical reasons that will simplify the exposition later on. This convention implies that the remainder of the division is a polynomial of degree $\leq d$ instead of degree $\leq d-1$.\medskip

Note also that in assertion (b) the map $\psi$ does depend on $y$ since the divisor $g(y)$ in $w=v+z= g(y)\cdot a+z$ depends on $y$. The map $\chi$ is only rationally textile as can be seen by taking $g(\y)=\y$, $y=c\cdot t^d$, $c\in\kkk^*$, and $w=t^d$. Then $a={1\over c}$ is rational in the coefficient of $y$.\medskip


We will apply the Weierstrass division in the case where the series $w$ which is divided by $g(y)$ varies over the components $y_j$ of the vector $y$ itself. In this case, it will be possible to determine explicitly the sets of quotients and remainders of the division as $y$ varies in $\SS_d$, see Theorem \ref{division}. This description will become crucial for the proof of the factorization theorem \ref{factorization}. We first need an auxiliary result. Recall that $\BB$ denotes one of the rings $\kkk[[\t,\y]]$, $\kkk\{\t,\y\}$ or $\kkk\langle\t,\y\rangle$. For later applications we write $\z$ instead of $\y$.


\begin{lemma}\label{principalideal} Let $g\in\BB$ be a formal power series in variables $\t$ and $\z=(\z_1,...,\z_m)$. For every choice of vectors $\zz$ in $\AAA^m$ and $\vv=g(\zz)\cdot a\in \langle g(\zz)\rangle\cdot\AAA^m$, the ideals of $\AAA$ generated by $g(\zz)$ and $g(\zz+\vv)$ coincide,
$$\langle g(\zz)\rangle =\langle g(\zz+\vv)\rangle.$$
\end{lemma}


\begin{rmk} The lemma implies that, whenever $g(\zz)$ is non-zero, the quotient $g(\zz)/g(\zz+g(\zz)\cdot a)$ is an invertible power series in $\mathcal A$. It is then easy to see that, denoting by $\SS\subset\AAA^m$ the set of vectors $\zz$ such that $g(\zz)\neq 0$, the map
$$\SS\times \AAA^m\map \AACC,$$
$$(\zz,a)\map {g(\zz)\over g(\zz+g(\zz)\cdot a)},$$
is \an \arquile map.
\end{rmk}

\begin{proof} Write $\vv=g(\zz)\cdot a$ for some $a\in\AAA^m$. Taylor expansion gives
$$g(\zz+\vv)=g(\zz+g(\zz)\cdot a)=g(\zz)+ r(\zz,g(\zz)\cdot a)$$ 
with some power series $r$ in $\t$, $\z$ and new variables $\w=(\w_1,...,\w_m)$ so that $r$ is at least of order $1$ in $\w$. Factoring $g(\zz)$ from the right hand side we get $g(\zz+v)=g(\zz)\cdot u(\zz,a)$ for some series $u$ in $\t$, $\z$ and $\w$. The series $u(\z,\w)$ is invertible since its constant term is $1$. As $\zz$ and $a$ vanish at $\t=0$, the evaluation $u(\zz,a)$ of $u$ is again invertible, now as a series in $\AACC$. The required equality $\langle g(\zz)\rangle=\langle g(\zz+\vv)\rangle$ of ideals in $\AAA$ follows.
\end{proof}


\begin{theorem} \label{division} Let $g\in\BB$ be a power series in $\t$ and $\y$. Denote by $\RR_d\subset\AAA^m$ the $\kkk$-vector space of polynomial vectors in $\t$ of degree $\leq d$ and without constant term, 
$$\RR_d=(\k[\t]_{\leq d})^m,$$ 
and by $\VV_d=\langle\t\rangle^d\cdot\AAA^m$ the $\AAA$-submodule of $\AAA^m$ generated in each component by $\t^d$. Denote furthermore by $\ZZ_d$ the intersection $\ZZ_d=\SS_d\cap \RR_d$, with $\SS_d=\{y\in\AAA^m,\, g(y)\neq 0,\, \ord\, g(y)=d\}$ as above. \medskip


{\rm (a)} For every vector $y\in\SS_d$ there exist unique vectors $a\in\AAA^m$ and $\zz\in\ZZ_d$ such that
$$y=g(y)\cdot a+\zz$$
is in each component of $y$ the Weierstrass division of $y_j$ by $g(y)$.


{\rm (b)} The map
$$\psi_d:\SS_d\map \ZZ_d\times \VV_d,$$
$$y\map (\zz,\vv),$$
given by the division $y=\vv+\zz=g(y)\cdot a +\zz$ is a textile isomorphism, and
$$\wt \psi_d:\SS_d\map \ZZ_d\times \AAA^m,$$
$$y\map (\zz,a),$$
is a rationally textile isomorphism.

\end{theorem}


\begin{rmks} (a) It is peculiar that the remainder $\zz\in\RR_d$ of the division of $y\in\SS_d$ by $g(y)$ actually lies again in $\SS_d$, i.e., satisfies $\ord\, g(\zz)=d$. This makes it possible to describe the images of $\psi_d$ and $\wt\psi_d$ as cartesian products: the first factor $\ZZ_d$ is a quasi-affine subvariety of the finite dimensional affine space $\RR_d$, the second factor $\VV_d$, respectively $\AAA^m$, is a free $\AAA$-module of finite rank.\medskip

(b) Formulas for the coefficients of the expansions of $\zz$, $\vv$ and $a$ can be obtained from the coefficients of $y$ by applying the usual division algorithm given by repeated substitutions analogous to the Euclidean division of polynomials.  \medskip

(c) In the case $\AAA=\k\{\t\}_s$, one has to restrict as in the preceding theorem to sufficiently small radii $s>0$, and vectors $y$ of sufficiently small $s$-norm.\end{rmks}


\begin{proof} For assertion (a), the Weierstrass division theorem \ref{weierstrass} ensures the existence and uniqueness of $\vv\in \VV_d$ and $\zz\in\RR_d$. So we only have to show that $\zz\in\SS_d$. But this is just the content of the preceding lemma, whence (a). Let us prove assertion (b). The maps $\psi$ and $\wt\psi$ are textile, respectively rationally textile maps by Theorem \ref{weierstrass}. The map 
$$\ZZ_d\times\VV_d\map \AAA^m:(\zz,\vv)\map \vv+\zz$$
sends, again by the lemma, $\ZZ_d\times\VV_d$ into $\SS_d$, and is therefore the inverse of $\psi$. Finally, as $\zz\in \ZZ_d$, the ideal of $\AAA$ generated by $g(\zz)$ equals $\langle \t\rangle^d$. We may therefore write $\vv\in\VV_d$ as $\vv=g(\zz)\cdot b$ with a unique $b\in\AAA^m$. The map associating $b=\vv/g(\zz)$ to a pair $(\zz,\vv)\in\ZZ_d\times\VV_d$ is clearly rationally textile. By Lemma \ref{principalideal}, both $g(\zz)$ and $g(\vv+\zz)$ generate the same ideal in $\AAA$, hence there is a unique $a\in\AAA^m$ such that $\vv=g(\vv+\zz)\cdot a $. Setting $y=\vv+\zz$ we see that $\vv=g(y)\cdot a$, say, $\wt\psi(y)= (\zz,a)$ and $\wt\psi$ is bijective. As we have
$$a={\vv\over g(y)}={\vv\over g(\vv+\zz)} = {\vv\over g(\zz)}\cdot {g(\zz)\over g(\vv+\zz)},$$
the map associating $a$ to a pair $(\zz,\vv)$ is again rationally textile. This gives (b) and finishes the proof.
\end{proof}



(2) {\it Division of deformations} :  To have an isomorphism $\psi_d$ as in assertion (b) of the preceding theorem it was necessary to restrict the set of vectors $y$ to the stratum $\SS_d$. Otherwise, for arbitrary $y$, the order of $g(y)$ may jump, so that the space of remainders could not be defined uniformly. Instead of working with a stratification and the cartesian product decomposition $\SS_d\isom\ZZ_d\times\VV_d$ of the stratum $\SS_d$, one may also perform the division for series whose coefficients lie in a complete local ring, i.e., for {\it deformations} of given power series vectors in $\AAA^m$. In this case, no stratification is necessary. This kind of division will be explained in the second part of this section. First we need a couple of definitions.
\medskip

Throughout the remainder of this section, only formal power series will be considered, so that $\AAA=\k[[\t]]$, $\BB=\kkk[[\t,\y]]$. The convergent case can be treated similarly, but requires considerably more efforts, cf.~Thme.~2, A.II.4, in \cite{Ha1}.\medskip

Let $S$ be a local ring, not necessarily Noetherian, with maximal ideal $\mm_S$ and residue field $S/\mm_S=\kkk$. It will always be equipped with a filtered topology defined by a decreasing sequence of ideals $J_k$ satisfying $J_k\cdot J_\ell\subset J_{k+\ell}$. We assume that $\mm_S^k\subset J_k$ for all $k$ and that $S$ is {\it complete}. Further, to avoid complications, we assume that $\kkk$ embeds into $S$ as a field of coefficients.%
\footnote{ In the Noetherian case, we typically take the $\mm_S$-adic topology on $S$; the non-Noetherian case is more subtle as is illustrated by the formal power series ring in countably many variables.}
Let us denote by $\AA_{S,\circ}=S_\circ[[\t]]$ the ring of formal power series in $\t$ with coefficients in $S$ and without constant term (in $S$). Series in $\AA_{S,\circ}$ will be distinguished from series in $\AAA$ by a tilde.  Taking the coefficients of a power series vector $\wt y\in\AA_{S,\circ}^m$ modulo $\mm_S$ gives a power series vector $y$ in $\AAA^m$, the {\it reduction} of $\wt y$ modulo $\mm_S$. A {\it deformation} of a vector $y\in\AAA^m$ over $S$ is, by definition, a vector $\wt y\in \AA_{S,\circ}^m$ whose reduction modulo $\mm_S$ equals $y$. The set of all deformations of $y$ parametrized by $S$ is denoted by $(\AA_{S,\circ}^m,y)$.\footnote{ The notation $(\AA_{S,\circ}^m,y)$ shall allude to the concept of the ``germ'' of $\AA_{S,\circ}^m$ at $y$, even though deformations do not refer to the \arquile or textile topology on $\AA_{S,\circ}^m$.} \medskip


\begin{theorem}  [Weierstrass division for deformations] \label{weierstrassdeformations} We restrict to the case of formal power series. For a given power series $g\in\BB$ denote by $\ww=g(y)\cdot a+\zz$ the Weierstrass division of a series $\ww\in \AAA$ by $g(y)$ for some $y\in\AAA^m$ for which $g(y)$ is non-zero of order $d$, with $a\in\AAA$ and $\zz\in\k[\t]_{\leq d}$. Let $S$ be a complete local ring as before with maximal ideal $\mm_S$ and residue field $S/\mm_S=\kkk$ embedding into $S$. 


{\rm (a)} For every vector $\wt y \in \AA_{S,\circ}^m$ and every series $\wt \ww\in \AA_{S,\circ}$, deformations of $y$ and $\ww$ respectively, there exist a unique power series $\wt a \in \AA_{S,\circ}$ and a unique polynomial $\wt\zz \in S[\t]_{\leq d}$ of degree $\leq d$ without constant term, both deformations of $a$ and $\zz$ respectively, such that 
$$\wt \ww = g(\wt y) \cdot \wt a +\wt \zz.$$
%


{\rm (b)} The coefficients of $\wt a$ and $\wt \zz$ are the evaluations in the coefficients of $\wt y$ and $\wt \ww$ of formal power series.

\end{theorem}


The statement also holds for convergent series but is then more delicate to prove, see Thme.~2, A.II.4, in \cite{Ha1}.


\begin{proof} We  show that the map
$$\a:(\AA_{S,\circ}^m,a)\times (S[\t]_{\leq d},\zz) \map (\AA_{S,\circ},\ww)$$
$$(\wt a, \wt \zz)\map \wt \ww = g(\wt y) \cdot \wt a +\wt \zz$$
is an isomorphism. Using that $\kkk\subset S$, we write $\wt y= y+y^\circ$ with $y^\circ\in \mm_S\cdot \AA_{S,\circ}^m$ and accordingly $g(\wt y)= g(y)+h(y,y^\circ)$ with $h\in\k[[\t,\y,\y^\circ]]$, where $\y^\circ=(\y_1^\circ,...,\y_m^\circ)$ denotes a new vector of variables. Observe here that $h(y,y^\circ)\in  \mm_S\cdot\AA_{S,\circ}$. By the assumption on $y$ we can write $g(y)=\t^d\cdot u$ with a unit $u\in \kkk[[\t]]$. 
Without loss of generality we may assume that $u=1$, say $g(y)=\t^d$. The map 
$$ \gamma(\wt a, \wt \zz)\map [ g(y)-g(\wt y)] \cdot \wt a$$
is the multiplication of $\wt a$ with $\t^d-g(\wt y)= g(y)-g(\wt y)=h(y,y^\circ)\in\mm_S\cdot\AA_{S,\circ}$. The map $\b$ given by $\b(\wt a, \wt\zz)= \t^d\cdot \wt a + \wt\zz$ is clearly an isomorphism. It therefore suffices to show that 
$$\a\circ\b^{-1} ={\rm Id}_{(\AA_{S,\circ},\ww)} - \gamma\circ\b^{-1}$$
is an isomorphism. To see this, observe first that the map $\gamma\circ\b^{-1}$ sends the maximal ideal $\mm_S\cdot\AA_{S,\circ}$ of $\AA_{S,\circ}$ into $\mm_S^2\cdot\AA_{S,\circ}$. Therefore the geometric series in $\gamma\circ\b^{-1}$, say
$$\delta=\sum_{k=0}^\infty (\gamma\circ\b^{-1})^k, $$
induces, by the completeness of $S$, a well defined map $\delta:\AA_{S,\circ}\map \AA_{S,\circ}$ which is the inverse to $\a\circ\b^{-1}$. This proves the theorem.
\end{proof}


\begin{rmk} In the above division, the coefficients of the expansion of $\wt w$ and $\wt y$ with respect to $\t$ are treated as variables, so the division is in this sense {\it universal}. One may then evaluate the variables by taking any elements of the maximal ideal of a complete local ring $S$ in order to obtain a corresponding division for deformations of $w$ and $y$ over $S$. \end{rmk}


We will apply the preceding theorem to the division of the components $\wt y_j$ of a deformation $\wt y$ of $y\in\AAA^m$ by $g(\wt y)$ analogously as we did in the first part of this section with $y$ itself. Let $y\in \AAA^m$ be given with $g(y)\neq 0$ of order $d$, and induced division $y=g(y)\cdot a+\zz$ with $a\in\AAA^m$ and $\zz\in\ZZ_d$ as in Theorem \ref{division}.  


\begin{theorem}  \label{divisiondeformations}\medskip

For a given power series $g(\t,\y)\in\BB$ denote by $y=g(y)\cdot a+\zz$ the componentswise Weierstrass division of a vector $y\in \AAA^m$ by $g(y)$ for some $y$ for which $g(y)$ is non-zero of order $d$, with $a\in\AAA^m$ and $\zz\in(\k[\t]]_{\leq d})^m$.\medskip


{\rm (a)} For every deformation $\wt y \in (\AA_{S,\circ}^m,y)$ of $y$ there exist a unique power series vector $\wt a \in \AA_{S,\circ}^m$ and a unique polynomial vector $\wt\zz \in (S[\t]_{\leq d})^m$ without constant term, both deformations of $a$ and $\zz$ respectively, such that 
$$\wt y = g(\wt y) \cdot \wt a +\wt \zz.$$
%

{\rm (b)} The coefficients of $\wt a$ and $\wt \zz$ are formal power series in the coefficients of $\wt y$.
\end{theorem}


\begin{proof} Both assertions follow from Theorem \ref{weierstrassdeformations}.
\end{proof}


\begin{rmks} (a) The statement is not a local version of the classical division theorem \ref{division} since the deformations of $y$ are not confined to the stratum $\SS_d$.\medskip

(b) The division does not hold for local rings $S$ which are not complete. It suffices to divide $\t^d$ by the deformation $\t^d+s_1\cdot\t^{d-1}+s_2\cdot \t^{d+1}$ of $\t^d$, with $s_1,s_2\in \mm_S$, to see that the remainder has as coefficients genuine power series in elements of $\mm_S$.
\end{rmks}


\section{Division modules}

This section can be skipped on first reading. The proofs of the factorization and linearization theorems are based on the decomposition of a power series vector $y\in\AAA^m$ into a sum $y=\vv+\zz$ where $\zz$ belongs to a finite dimensional space of polynomial vectors and where $\vv$ is the part of $y$ on which the linearization can be achieved. The decomposition is done by the componentswise Weierstrass division of $y$ by a suitable divisor. This section characterizes possible choices of such divisors.


\begin{definition}\label{divisionmoduledef} A {\it division module} for a vector $f=f(\t,\y)\in \BB^\kk$ is a $\BBB$-submodule $\GG$ of $\BBB^m$, with $\BBB=\BB\cap \k[[\t,\y]]$, which is a cartesian product
$$\GG= \langle g_1\rangle \times \ldots\times \langle g_m\rangle \subset\BBB^m$$
of  non-zero principal ideals $\langle g_i\rangle :=g_i\cdot \BBB$ of $\BBB$, with $g_i\in\BB$, satisfying the following conditions:


\begin{enumerate}

\item[(i)] For all $\zz\in\AAA^m$ and $\vv\in \GG(\zz):=
g_1(\zz)\cdot\AAA \times \ldots\times g_m(\zz)\cdot\AAA\subset\AAA^m$ we have
$$\GG(\zz)=\GG(\zz+\vv).$$
%
%
\item[ii)] Write elements $\vv$ of $\GG(\zz)$ as $\vv=\DD_g(\zz)\cdot a$ with column vector $a\in\AAA^m$ and a diagonal matrix $\DD_g=\diag(g_1,...,g_m)\in\BB^{m\times m}$. There exists a (formal, convergent, respectively algebraic) power series vector $p\in\kkk[[\t,\z,\aa]]^m$ in three sets of variables $\t$, $\z=(\z_1,\ldots,\z_m)$ and $\aa=(\aa_1,\ldots,\aa_m)$, which is at least of order $2$ in $\aa$ and satisfies for all $\zz$, $a\in\AAA^m$
$$f(\zz+\DD_g(\zz)\cdot a)=f(\zz)+\6_{\y}f(\zz)\cdot \DD_g(\zz)\cdot a+\6_{\y}f(\zz)\cdot \DD_g(\zz)\cdot p(\zz,a).$$
\end{enumerate}
\end{definition}\goodbreak
%
%
\begin{rmks} \label {rmkdivisionmodule} (a) Condition (i) can be rewritten as an equality of ideals in $\AAA$, say $\langle g_i(\zz)\rangle =\langle g_i(\zz+\DD_g(\zz)\cdot a)\rangle$, for all $\zz$ and $a$ in $\AAA^m$, and $i=1,...,m$. Phrased differently, there exists an invertible matrix $U\in\Gl_m(\kkk[[\t,\z,\aa]])$ (with formal, convergent, respectively algebraic entries) such that
$$\DD_g(\zz) = \DD_g(\zz+\DD_g(\zz)\cdot a)\cdot U(\zz,a)$$
holds for all $\zz$ and $a$.\medskip

(b) Condition (i) appears in Lemma \ref{principalideal}, which is used in the proof of Theorem \ref{division}, and in the proof of Proposition \ref{isomorphisms}.\medskip

(c) Condition (ii) ensures that the terms of order $\geq 2$ of the Taylor expansion of $f(\zz+\DD_g(\zz)\cdot a)$ belong to the module generated by the linear terms $\6_{\y}f(\zz)\cdot \DD_g(\zz)\cdot a$. This module equals the $\AAA$-submodule $\6_{\y}f(\zz)\cdot \GG(\zz)$ of $\AAA^\kk$ which is the image of $\GG(\zz)$ under the  tangent map at $\zz$ of the \arquile map $\ff:\AAA^m\map\AAA^\kk$ induced by $f$. In the terminology of [BH], condition (ii) signifies that for each $\zz \in\AAA^m$, the map 
$$\a_\zz:\GG(\zz)\map \AAA^\kk,\,\vv=\DD_g(\zz)\cdot a\map f(\zz+\DD_g(\zz)\cdot a)$$
is a {\it quasi-submersion}, i.e., that its image is contained in the image of the tangent map. More generally speaking, it is a map of {\it constant rank} in the sense of \cite{HM}. This is a necessary condition to be ``locally linearizable''. To make this precise and to prove that with the appropriate assumptions this is also a sufficient condition, that is, that the map $\a_\zz$ is indeed linearizable (in a specific sense), will be the subject of section \ref{section_linearization}.\medskip

(d) Note that $\DD_g(\zz)\cdot a$ is just the vector $(g_1(\zz)\cdot  a_1,...,g_m(\zz)\cdot a_m)\in \AAA^m$ and that the series $g_i$ may have non-zero constant term $g_i(0)$. \medskip

(e) The name ``division module'' is motivated by the later decomposition of vectors $y\in\AAA^m$ into $y=\DD_g(y)\cdot \wt a+\zz=\DD_g(\zz)\cdot a+\zz$ given by the Weierstrass division of the components $y_i$ by $g_i(y)$, see Proposition \ref{isomorphisms}.
\end{rmks}


{\bf Example.} Typically, one could take for $\GG$ the $m$-fold cartesian product 
$$\GG=\langle g^{o_1}\rangle\times \ldots\times \langle g^{o_m}\rangle$$
of ideals of $\BBB$ for some chosen non-zero series $g\in\BB$ and non-negative integers $o_i$. It is easy to see that such a $\GG$ satisfies condition (i) of the definition. However, condition (ii) will not hold without extra assumptions on $g$. These are specified in the next result.
\medskip


\begin{prop} \label{divisionmodule} Let $f\in \BB^\kk$ be given. Assume that $\kk\leq m$ and that some $(\kk\times\kk)$-minor $g$ of the relative Jacobian matrix $\6_\y f$ of $f$ is not zero. For any choice of non-negative integers $o_1,...,o_m$ satisfying
$$o_i+o_j\geq o_l+1$$
for all $i,j\in\{1,\cdots,m\}$ and $\ell\in\{1,...,\kk\}$ the $\BBB$-module 
$$\GG=\langle g^{o_1}\rangle\times\ldots\times\langle g^{o_m}\rangle\subset\BBB^m$$
is a division module for $f$. 
\end{prop}
%

\begin{rmk} This proposition will be used later on in the case where $o_i=1$ for $1\leq i\leq \kk$ and $o_i=2$ for $\kk+1\leq i\leq m$.
\end{rmk}

\begin{proof} For condition (i), it suffices to show, since $\GG$ is a cartesian product of principal ideals, that for any $\zz\in\AAA^m$ and $\vv\in \GG(\zz)$ we have the equality of ideals 
$$\langle g(\zz)\rangle=\langle g(\zz+\vv)\rangle$$
in $\AAA$. This equality does not use that $g$ is a minor -- it is the content of Lemma \ref{principalideal}.\medskip


To prove that $\GG$ fulfills condition (ii), let $\vv=\DD_g(\zz)\cdot a\in \GG(\zz)$ with $a\in\AAA^m$. Taylor expansion gives 
$$f(\zz+\DD_g(\zz)\cdot a)=f(\zz)+\6_\y f(\zz)\cdot \DD_g(\zz)\cdot a + q(\zz,\DD_g(\zz)\cdot a)$$
with some vector $q\in \kkk[[\t,\z,\w]]^\kk$ which is of order at least two in $\w=(\w_1,...,\w_m)$. We have to find a power series vector $p\in\kkk[[\t,\z,\aa]]^\kk$ of order at least two in $\aa=(\aa_1,...,\aa_m)$ such that
$$q(\zz,\DD_g(\zz)\cdot a) = \6_y f(\zz)\cdot \DD_g(\zz)\cdot p(\zz,a)$$ 
holds for all $\zz$ and $a$ in $\AAA^m$. First observe that each component of $q(\zz,\DD_g(\zz)\cdot a)$ belongs to the ideal of $\AAA$ generated by  the powers $g(\zz)^{o_i+o_j}$ for $i,j\in\{1,...,m\}$. It therefore suffices to show that any vector of the form $(0,...,0,g^{o_i+o_j},0,...,0)$, with $i,j\in\{1,\cdots,m\}$, can be written as a $\BBB$-linear combination of the vectors $\6_{\y_1}f\cdot g^{o_1},...,\6_{\y_\kk}f\cdot g^{o_\kk}$. \medskip

To prove this, set $\y^1=(\y_1,...,\y_\kk)$, so that $g$ is without loss of generality the determinant of the $(\kk\times \kk)$-submatrix $\6_{\y^1}f$ of $\6_\y f$. Denote by $\6^*_{\y^1}f$ the adjoint matrix of $\6_{\y^1}f$. We then have $\6_{\y^1}f\cdot\6^*_{\y^1}f=g\cdot\1_\kk$, and accordingly $g^o\cdot \6_{\y^1}f\cdot \6^*_{\y^1}f=g^{o+1}\cdot\1_\kk$ for any integer $o$. 
Taking now for $o$ the values $o_i+o_j-1$ and using the assumption that $o_i+o_j\geq o_l+1$ for all $i,j\in\{1,\cdots,m\}$ and $\ell\in\{1,...,\kk\}$ we get the claim. This establishes condition (ii).
\end{proof} 


\section{Linearization theorem}\label{section_linearization}

This is the central section of the paper. To prepare the setting of the linearization theorem below, let $f=(f_1,...,f_\kk)\in\BB^\kk$ be a vector of power series which admits a division module $\GG=\langle g_1\rangle\times \ldots\times \langle g_m\rangle\subset \BBB^m$. Recall that by Proposition \ref{divisionmodule} the existence of such a module is ensured for instance in the case where $\kk\leq m$ and if there exists a non-zero $(\kk\times\kk)$-minor $g$ of $\6_\y f$. Up to a permutation of the variables we may assume that $g$ is given by the first $\kk$ columns of $\6_\y f$. By Proposition \ref{divisionmodule}, the $\BBB$-module $\GG=\langle g^{o_1}\rangle\times\ldots\times\langle g^{o_m}\rangle$ with  $o_i+o_j\geq o_l+1$ for all $i,j\in\{1,\cdots,m\}$ and $\ell\in\{1,...,\kk\}$ is a division module for $f$. \medskip

The following constructions (until the next lemma) hold for any vector $g\in \BB^m$. For a chosen $m$-tuple $\dd=(d_1,...,d_m)\in\N^m$ set
$$\SS_\dd=\{y\in \AAA^m,\, g_i(y)\neq 0,\, \ord\,g_i(y)=d_i \text { for all } i\}.$$
Since $\t$ is just a single variable, the condition $\ord\,g_i(y)=d_i$ is equivalent to saying that $g_i(y)$ equals the monomial $\t^{d_i}$ up to the multiplication by a unit in $\mathcal A$, i.e., that $\langle g_i(y)\rangle=\langle\t^{d_i}\rangle$ as ideals of $\AAA$, for each $y\in\SS_\dd$. This implies that $\SS_\dd$ is a cofinite \ttextile locally closed subset of $\AAA^m$. It is $\t$-adically open since $\SS_\dd$ contains with every element $y$ the $\t$-adic neighborhood $y+\langle \t^{d_1}\rangle\times\ldots\times \langle \t^{d_m}\rangle$ in $\AAA^m$. We agree to set $\ord\,g_i(y)=\infty$ if $g_i(y)=0$ and then define accordingly $\SS_\dd$ also if some $d_i=\infty$. This gives a countable partition of $\AAA^m$ into
$$\AAA^m=\bigcup_{\dd\in(\N\cup\{\infty\})^m} \SS_\dd.$$
Restrict now to $m$-tuples $\dd\in\N^m$ and vectors $y\in \SS_\dd$. We set
$$\VV_\dd=\GG(y)=\langle\t^{d_1}\rangle\times \ldots\times \langle\t^{d_m}\rangle\subset\AAA^m.$$
We consider $\VV_\dd$ as an $\AAA$-submodule of $\AAA^m$. Let $\RR_\dd$ be the direct monomial complement of $\VV_\dd$ in $\AAA^m$ consisting of vectors $\zz\in\AAA^m$ whose $i$-th component $\zz_i$ is a polynomial in $\k[\t]$ of degree $\leq d_i$,
$$\RR_\dd=\k[\t]_{\leq d_1}\times \ldots\times \k[\t]_{\leq d_m},$$
with direct sum decomposition
$$ \VV_\dd\oplus\RR_\dd=\AAA^m.$$
Accordingly, we decompose vectors $y\in\SS_\dd$ into $y=\vv+\zz$ 
with $\vv\in \VV_\dd$ and $\zz\in\RR_\dd$.  The $i$-th component $\zz_i$ of $\zz$ is the remainder of the Weierstrass division of the $i$-th component $y_i$ of $y$ with respect to $\t^{d_i}$, cf. Theorem \ref{division}. Denote by 
$$\pi_\dd:\SS_\dd\map \RR_\dd,\,y\map z,$$
the induced textile map. Set 
$$\ZZ_\dd=\SS_\dd\cap \RR_\dd=\{\zz\in \RR_\dd,\, g_i(\zz)\neq 0 \text{ and }\ord\,g_i(\zz)=d_i \text { for all } i\}.$$
Note that $\RR_\dd$ is a finite dimensional $\kkk$-vector space of dimension $d_1+\ldots +d_m$ which will be considered as affine space $\RR_\dd\isom\A_\kkk^{d_1+\ldots +d_m}$, and $\ZZ_d$ is a Zariski locally closed subset. Both $\VV_\dd$ and $\RR_\dd$ only depend on $\dd=(d_1,...,d_m)$. The sets $\SS_\dd$ and $\ZZ_\dd$ will depend on $f$ in the situation of Proposition \ref{divisionmodule} where the series $g_i=g^{o_i}$ are chosen as powers of a minor $g$ of $\6_\y f$. 
\goodbreak



\begin{lemma}\label{image} Let $f\in\BB^\kk$ be a power series vector with division module $\GG\subset\BBB^m$, and let $\dd\in\N^m$. The associated projection map $\pi_\dd:\SS_\dd\map \RR_\dd,\,y\map z$, has image $\ZZ_\dd$.
\end{lemma}

Said differently, if $y\in\AAA^m$ satisfies $\ord\, g_i(y)=d_i$ for all $i$, then its remainder $z$ satisfies the same condition. 

%
%
 %


\begin{proof} 
Write $y=\vv+\zz$ as above. By def.~\ref{divisionmoduledef} (i) of division modules we know that $\GG(\zz)=\GG(y-v)=\GG(y)$ so that $\zz\in \SS_\dd$. This proves that the image of $\pi$ is included in $\ZZ_\dd=\SS_\dd\cap\RR_\dd$. It is whole $\ZZ_\dd$ since the restriction of $\pi_\dd$  to $\ZZ_\dd$ is the identity map.  \end{proof}
%
%


\begin{prop}\label{isomorphisms} Let $f\in\BB^\kk$ be a power series vector with division module $\GG\subset\BBB^m$, diagonal matrix $\DD_g$, associated sets $\SS_\dd$, $\VV_\dd$, $\RR_\dd$ and $\ZZ_\dd$, and projection $\pi_\dd:\SS_\dd\map \ZZ_\dd,\,y\map z$, as defined above, for $\dd\in\N^m$.\medskip

{\rm (a)} The map 
$$\psi_\dd:\SS_\dd\map \ZZ_\dd\times\VV_\dd, \, y=\vv+\zz\map (\zz,\vv),$$
given by the division of $y$ by $\GG(y)$ is a textile isomorphism, restricting to the identity on $\ZZ_\dd$, with inverse  
$$\psi_\dd^{-1}:\ZZ_\dd\times\VV_\dd\map\SS_\dd,\, (\zz,\vv)\map\vv+\zz.$$
%
%
{\rm (b)} For $(\zz,\vv)\in\ZZ_\dd\times\VV_\dd$, let $a\in\AAA^m$ be the unique vector such that $\vv=\DD_g(\zz)\cdot a$. Let $p\in \kkk[[\t,\z,\aa]]^m$ be a vector which is of order at least $2$ in $\aa$. The map
$$\phi_\dd: \ZZ_\dd\times\VV_\dd\map \ZZ_\dd\times\VV_\dd,$$
$$(\zz,\vv)\map  (\zz, \DD_g(\zz)\cdot (a+p(\zz,a)))$$
is a rationally textile isomorphism, restricting to the identity on $\ZZ_\dd\times 0$. 
\end{prop}


\begin{rmk} Geometrically speaking, assertion (a) says that the stratum $\SS_\dd$ can be considered as a ``bundle'' over $\ZZ_\dd$, with fiber $\VV_\dd$. This will become relevant in the linearization of \arquile maps along $\SS_\dd$, see the proof of Theorem \ref{linearization} below.
\end{rmk}


\begin{proof} For assertion (a), notice first that the map $\psi_\dd$ is well defined by the preceding lemma. It is clearly injective. We claim that the image of $\psi_\dd$ equals $\ZZ_\dd\times\VV_\dd$. We have already seen that the first component of $\psi_\dd$ has image $\ZZ_\dd$. It remains to show that for all vectors $\ww\in \VV_\dd$ there exists a vector $y\in\SS_\dd$ with remainder $\zz$ such that $\ww=y-\zz$.
\medskip

So let $(\zz,\ww)$ be an element of $\ZZ_\dd\times\VV_\dd$, and set $y=\ww+\zz\in \AAA^m$.  As  $\zz\in\SS_\dd$ and $\ww\in \GG(\zz)$ condition (i) of division modules implies that $\GG(\zz)=\GG(\zz+\ww)$. This implies $\GG(y)=\GG(\ww+\zz)=\GG(\zz)=\VV_\dd$. Therefore  $y\in \SS_\dd$. 
Finally, since $\ww\in \GG(\zz)=\VV_\dd=\GG(y)$,  the decomposition $y=\ww+\zz$ is indeed the division of $y\in\SS_\dd$ by $\GG(y)$. It follows that $\psi_\dd(y)=(\zz,\ww)$, hence $\psi_\dd$ is onto. The inverse of $\psi_\dd$ is given by addition, $(\zz,\vv)\map \vv+\zz$, and is therefore a textile (and even an arquile) map. This shows that $\psi_\dd$ is a rationally textile isomorphism and hence also a $\t$-adic homeomorphism.\medskip


As for assertion (b) and the map $\phi_\dd$, Theorem \ref{inv} ensures that for every $\zz\in\AAA^m$ the map 
$$r_\zz: \AAA^m\map\AAA^m,\, a\map a+p(\zz,a)$$ 
is \an \arquile isomorphism. The map $\phi_\dd:\ZZ_\dd\times\VV_\dd\map\ZZ_\dd\times\VV_\dd$ is therefore a textile isomorphism and hence also a $\t$-adic homeomorphism.\end{proof}
\goodbreak

\begin{theorem} [Linearization of \arquile maps] \label{linearization}
Let $f\in\BB^\kk$ be a vector of power series in $\t$ and $\y$, with induced \arquile map
$$\ff:\AAA^m\map\AAA^\kk,\, y=y(\t)\map f(y)= f(\t, y(\t)).$$
Assume that $\kk\leq m$ and that there is a $(\kk\times\kk)$-minor $g$ of the relative Jacobian matrix $\6_\y f$ of $f$ which is not identically zero. \medskip

There then exist \an \arquile open dense subset $\SS$ of $\AAA^m$, a partition $\SS=\bigcup_{d\in\N}\, \SS_d$ into \ttextile locally closed sets $\SS_d$ and, for each $d\in\N$, textile isomorphisms $\chi_d:\ZZ_d\times \VV_d\map \SS_d$ over $\ZZ_d$ with $\ZZ_d\subset \A_\kkk^\ell$ a quasi-affine subvariety and $\VV_d\subset\AAA^m$ a free $\AAA$-submodule such that the composition
$$\ff\circ \chi_d:\ZZ_d\times \VV_d\map \AAA^\kk,$$
$$(\zz,\vv)\map f(\zz) + \6_\y f(\zz)\cdot \vv,$$
is linear in the second component $\vv$. \end{theorem}


\noindent In this situation, we say that arquile maps are {\it essentially linear} with respect to the chosen stratification $\SS=\bigsqcup \SS_d$ of $\AAA^m$. \medskip


The statement of the theorem can be made more precise and formulated in a slightly more general setting: It suffices to suppose that $f$ admits a division module $\GG=\langle g_1\rangle\times \ldots\times \langle g_m\rangle\subset \BBB^m$. Let then the diagonal matrix $\DD_g=\diag(g_1,...,g_m)$ and the vector $p\in\kkk[[\t,\aa,\zz]]^\kk$ be defined as in definition\ref{divisionmoduledef}. In this case, fix $\dd=(d_1,...,d_m)\in \N^m$ and let $\SS_\dd=\{y\in \AAA^m,\, {\rm ord}\,g_i(y)=d_i \text { for all } i\}$, $\VV_\dd=\langle\t^{d_1}\rangle\times \ldots\times \langle\t^{d_m}\rangle\subset\AAA^m$ with direct complement $\RR_\dd=\k[\t]_{\leq d_1}\times \ldots\times \k[\t]_{\leq d_m}$, and $\ZZ_\dd=\SS_\dd\cap \RR_\dd$ be defined as in the paragraphs preceding the theorem. 


\begin{theorem} [Linearization of \arquile maps, general case] \label{generallinearization} Consider the textile isomorphisms 
$$\psi_\dd:\SS_\dd\map \ZZ_\dd\times\VV_\dd, \, y=\vv+\zz\map (\zz,\vv)$$
and
$$\phi_\dd=(\phi_\dd^1,\phi_\dd^2): \ZZ_\dd\times\VV_\dd\map \ZZ_\dd\times\VV_\dd,\, (\zz,\vv)\map  (\zz,\DD_g(\zz)\cdot (a+p(\zz,a)))$$
from Proposition \ref{isomorphisms}, so that 
$$\phi_\dd\circ\psi_\dd:\SS_\dd\map \ZZ_\dd\times \VV_\dd,\, y=\vv+\zz=\DD_g(\zz)\cdot a+\zz\map (\zz,\DD_g(\zz)\cdot (a+p(\zz,a))).$$
Define the textile isomorphism $\chi_\dd$ as the inverse
$$\chi_\dd=(\phi_\dd\circ\psi_\dd)^{-1}:\ZZ_\dd\times\VV_\dd\map\SS_\dd.$$
Then the composition $\ff\circ \chi_\dd$ is linear in $\vv\in\VV_\dd$ of the form
$$\ff\circ \chi_\dd:\ZZ_\dd\times\VV_\dd\map \AAA^\kk,$$
$$(\zz,\vv)\map f(\zz)+\6_\y f(\zz)\cdot \vv.$$

\end{theorem}
%

\begin{rmks}\label{spaces} (a) A typical instance of this situation is the case where the division module $\GG$ of $f$ has the form $\GG=\langle g^{o_1}\rangle\times\ldots\times\langle g^{o_m}\rangle\subset\BBB^m$ for $g$ a suitable $(\kk\times\kk)$-minor of $\6_\y f$ and integers $o_i$ as in Proposition \ref{divisionmodule}. In this case, one would set $\dd=(o_1\cdot d,...,o_m\cdot d)$ for $d\in\N$ and then define $\SS_d=\SS_\dd$, $\VV_d=\VV_\dd$ accordingly as at the beginning of the section. This will allow us to see the statement of the first theorem as a special case of the second one.\medskip


(b) Notice that $\ff$ is linearized with respect to $\vv$ only on the strata $\SS_\dd$, but not on whole $\AAA^m$. Again, one should view $\SS_\dd$ here as a bundle over $\ZZ_\dd$ with fiber $\VV_\dd$, the linearization taking place on the second factor $\VV_\dd$, and depending on the parameters $\zz$ in $\ZZ_\dd$.\medskip


(c) The map $\chi_\dd=(\phi_\dd\circ\psi_\dd)^{-1}=\psi_\dd^{-1}\circ\phi_\dd^{-1}:\ZZ_\dd\times\VV_\dd\map\SS_\dd$ is ``almost'' given by substitution: The only drawback is that for $\phi_\dd^{-1}$ we first have to divide the second component $\vv$ of an element $(\zz,\vv)$ of $\ZZ_\dd\times\VV_\dd$ by $\DD_g(\zz)$ in order to be then able to invert $a\map a+p(\zz,a)$. The second map $\psi_\dd^{-1}$ is just the addition $(\zz,\vv)\map\vv+\zz$.\medskip


(d) The statement of the theorem is a variant of the constant rank theorem for \arquile maps between power series spaces, see \cite{HM, BH}.
\end{rmks}


\begin{proof} We shall prove the stronger statement as given in Theorem \ref{generallinearization}. The assertions of Theorem \ref{linearization} then follow by taking for $f$ the division module $\GG=\langle g^{o_1}\rangle\times\ldots\times\langle g^{o_m}\rangle\subset\BBB^m$ as in Proposition \ref{divisionmodule}. We expand $f(y)=f(\vv+\zz)$ for $\zz\in\ZZ_\dd$ and $\vv=\DD(\zz)\cdot a\in\VV_\dd$ as follows, using condition (ii) of def.~\ref{divisionmoduledef}, 
$$(\ff\circ\psi_\dd^{-1})(\zz,\vv)=f(\vv+\zz)=f(\zz)+\6_\y f(\zz)\cdot \DD(\zz)\cdot a+\6_{\y}f(\zz)\cdot \DD_g(\zz)\cdot p(\zz,a)$$
$$ =f(\zz)+   \6_{\y}f(\zz)\cdot \DD_g(\zz)\cdot (a+p(\zz,a))$$
$$=f(\zz)+\6_{\y}f(\zz)\cdot \phi_\dd^2(\zz,\vv),$$
with $p$ and $\phi_\dd=(\phi_\dd^1,\phi_\dd^2)$ as before.
This gives
$$(f\circ\chi_\dd)(\zz,\vv)=(\ff\circ\psi_\dd^{-1}\circ\phi_\dd^{-1})(\zz,\vv)=f(\zz)+ \6_{\y}f(\zz)\cdot\vv$$
as required. The compatibility of all constructions under restriction to the convergent and algebraic setting is evident from the construction of the maps $\psi_\dd$ and $\phi_\dd$ in Proposition \ref{isomorphisms}. This proves the theorem.
\end{proof}

\begin{rmk} \label{proofartin} Let us briefly sketch how the linearization theorem provides an alternative proof of the Artin approximation theorem in the univariate case as was claimed in Corollary \ref{artin}. The details are explained in section 9 of \cite{Ha2}. We first treat the classical version as in part (a) of the corollary. So let $\wh y=\wh y(\t)\in\k[[\t]]^m$ be a formal solution of $f(\t,\y)=0$. We may enlarge this system of equations such that the components $f_i$ of $f$ generate the prime ideal $\P_{\wh y}$ of all convergent, respectively algebraic series $h$ vanishing at $\wh y$. This implies that there is a minor $g$ of the relative Jacobian matrix $\6_\y f$ of $f$ with respect to $\y$ of size the height of $\P_{\wh y}$ which does not vanish at $\wh y$.%
\footnote{ For the existence of such a minor one uses the equation $f(\t,\wh y(\t))=0$ and a reasoning as in the proof of Theorem \ref{singularlocus}.}
So we may assume that $\wh y$ is an \arqregular point of the \arquile variety $\YY\subset \k[[\t]]^m$ defined by $f$. Let $d$ be the order of $g(\t,\wh y(\t))$. Now apply the linearization theorem to the \arquile map $f_\infty: \k[[\t]]^m\map \kkk[[\t]]^\kk$ locally at $\wh y$. As $f$ is convergent, respectively algebraic, the trivializing textile isomorphism $\chi_d$ for $f_\infty$ will send the subspaces of convergent, respectively algebraic power series into themselves. After application of this isomorphism we may assume that $f$ is $\k[[\t]]$-linear up to a subspace of finite dimension. Now the density assertion follows from the flatness of $\k[[\t]]$ over $\k\{\t\}$, respectively $\k\langle\t\rangle$.%
\footnote{See Thm.~7.6 in \cite{Ma}. One may also invoke here the Artin-Rees lemma.}
\medskip

Let us now turn to the second part (b) of Corollary \ref{artin}. Let $I$ be the ideal of $\kkk\{\t,\y\}$, respectively $\kkk\langle\t,\y\rangle$ generated by the components of $f$. Let $J$ be the largest ideal of these rings which contains $I$ and which admits approximate solutions $\ol y$ up to any degree for all its elements. This is a prime ideal. We may then assume from the beginning that the components of $f$ generate this ideal, so that $I=J$. Let $r$ be its height. There then exists an $(r\times r)$-minor $g$ of $\6_\y f$ of $f$ which does not belong to $I$.%
\footnote{ The existence of such a minor along the lines of the proof of Theorem \ref{singularlocus} requires now to replace the equation $f(\t,\wh y(\t))=0$ by the vanishing of $f(\t,\ol y(\t))$ up to sufficiently high order.}
 Let $\ol y=\ol y(\t)$ be a $\t$-adically sufficiently good approximate solution of $f(\t,\y)=0$. Then $g(\t,\ol y(\t))$ does not vanish. Let $d$ be its order. We may now apply the linearization theorem to $f_\infty$ at $\ol y$. The trivializing textile isomorphism $\chi_d$ for $f_\infty$ will be a $\t$-adic homeomorphism. After application of this isomorphim we may assume that $f_\infty$ is $\kkk[[\t]]$-linear up to a subspace of finite dimension. Now the existence of a lifting of $\ol y$ follows for instance from the Weierstrass division theorem for formal power series, or the faithful flatness of $\k[[\t]]$ over $\k\{\t\}$ and $\k\langle\t\rangle$, see Thm.~7.6 in \cite{Ma}.\end{rmk}


\section{Fibration theorem}

As an immediate consequence of the strata-wise linearization of \arquile maps $\ff:\AAA^m\map\AAA^\kk$ we obtain the following description of the geometry of \arquile subvarieties $\YY(f)$ of $\AAA^m$.


\begin{theorem} [Fibration of arquile varieties] \label{fibration} Let $f\in\BB^\kk$ be a power series vector with $\kk\leq m$ and assume that there is a $(\kk\times\kk)$-minor $g$ of the relative Jacobian matrix $\6_\y f$ of $f$ which is not identically zero. Let the sets $\SS_d$, $\VV_d$, $\RR_d$ and $\ZZ_d$ for $d\in\N$ be defined as at the beginning of section \ref{section_linearization}. Let
$$\YY=\YY(f)=\{y\in\AAA^m,\, f(y)=0\}$$ 
be the zeroset of $f$ in $\AAA^m$, set $\YY_d=\YY_d(f)=\YY(f)\cap \SS_d$, and let $\pi_d:\SS_d \map \ZZ_d,\, y\map \zz$ be the projection map given by division as in Lemma \ref{image}. Introduce the sets
$$\WW_d=\{(\zz,\vv)\in\ZZ_d\times\VV_d,\, f(\zz)+\6_\y f(\zz)\cdot \vv=0\},$$
$$\ZZ_d^*=\{\zz\in\ZZ_d,\,f(z)\in \6_{\y}f(\zz)\cdot \VV_d\}.$$
%

{\rm (a)} The projection $\tau_d:\WW_d\map \ZZ_d^*,\, (\zz,\vv)\map \zz,$ on the first factor defines a linear fibration of $\WW_d$ over $\ZZ_d^*$ whose fibers $\WW_{d,\zz}=\tau_d^{-1}(\zz)$ are the affine $\AAA$-submodules of $\AAA^m$ of vectors $\vv\in\VV_d$ satisfying the equation 
$$f(\zz)+\6_\y f(\zz)\cdot \vv=0.$$


{\rm (b)} The map $\chi_d^{-1}=\phi_d\circ\psi_d: \SS_d\map \ZZ_d\times\VV_d$ induces by restriction to $\YY_d$ a textile isomorphism $\xi_d=(\phi_d\circ\psi_d)_{\vert \YY_d}:\YY_d\map\WW_d$ over $\ZZ_d^*$,  
%
%
$$\xymatrix{ \YY_d \ar[rr]^{\xi_d} \ar[rdd]_{\pi_d} & & \WW_d \ar[ldd]^{\tau_d}\\
& & \\
 & \ZZ_d^*  & }$$


{\rm (c)} The set
$$\WW_d=\{(\zz,\vv)\in\ZZ_d\times\VV_d,\, f(\zz)+\6_\y f(\zz)\cdot \vv=0\},$$
is arquile closed and the set
$$\ZZ_d^*=\{\zz\in\ZZ_d,\,f(z)\in \6_{\y}f(\zz)\cdot \VV_d\}$$
is Zariski closed in $\ZZ_d$ and hence a quasi-affine subvariety of some finite dimensional affine space $\A_\kkk^\ell$.
\end{theorem}


\begin{proof} Assertion (a) is immediate from the definition of $\WW_d$ and $\ZZ_d^*$, and (b) follows from the linearization theorem \ref{linearization}. Let us prove assertion (c). It is clear that $\WW_d$ is arquile closed. As for $\ZZ_d$, recall first that $\ZZ_d=\SS_d\cap\RR_d=\{\zz\in\RR_d,\, g(\zz)\neq 0$ and $\ord\, g(\zz)=d\}$ is a Zariski locally closed subvariety of the finite dimensional $\kkk$-vector space $\RR_d=\k[\t]_{\leq d}^m$. We have $\GG= \langle g^{o_1}\rangle\times \ldots\times \langle g^{o_m}\rangle$ and $\VV_d=\GG(\zz)=\langle \t^{o_1\cdot d}\rangle\times \ldots\times \langle \t^{o_m\cdot d}\rangle$ for $\zz\in \SS_d$. We may assume that $g$ is the $(\kk\times\kk)$-minor defined by the first $\kk$ columns of $\6_\y f$. Let $\y^1=(\y_1,...,\y_\kk)$ and $\y^2=(\y_{\kk+1},...,\y_m)$ denote the first $\kk$, respectively last $m-\kk$, components of $\y$, and write accordingly $\VV_d=\VV_d^1\times\VV_d^2\subset\AAA^\kk\times\AAA^{m-\kk}$. Let $\6_{\y^1}^*f$ be the adjoint matrix of the $(\kk\times \kk)$-matrix $\6_{\y^1}f$. Then $\6_{\y^1}^*f\cdot \6_{\y^1}f=g\cdot\1_\kk$. This implies that the module 
$$\6_{\y}f(\zz)\cdot \VV_d=\6_{\y^1}f(\zz)\cdot \VV_d^1+\6_{\y^2}f(\zz)\cdot \VV_d^2\subset \AAA^\kk$$
contains the submodule $g(\zz)\cdot\VV_d^1$. Observe that this submodule has finite codimension in $\AAA^\kk$ as a $\kkk$-vector space, since 
$$g(\zz)\cdot \VV_d=\langle g(\zz)^{o_1+1}\rangle\times \ldots\times \langle g(\zz)^{o_m+1}\rangle$$
equals for $\zz\in \ZZ_d$ the $\AAA$-submodule 
$$\langle \t^{(o_1+1)d}\rangle\times \ldots\times \langle \t^{(o_m+1)d}\rangle$$
of $\AAA^m$. Now, for a vector $\zz\in\ZZ_d$, the condition $f(\zz)\in \6_\y f(\zz)\cdot \VV_d$ defining $\ZZ_d^*$ in $\ZZ_d$ is equivalent to the condition 
$$\6_{\y^1}^*f(\zz)\cdot f(\zz)\in \6_{\y^1}^*f(\zz)\cdot \6_{\y}f(\zz)\cdot \VV_d$$
since, because of $g(\zz)\neq 0$, the linear map $\AAA^\kk\map \AAA^\kk$ induced by $\6_{\y^1}^*f(\zz)$ is injective. From $\VV_d=\VV_d^1\times\VV_d^2$ it follows 
$$\6_{\y^1}^*f\cdot \6_{\y}f\cdot \VV_d=\6_{\y^1}^*f\cdot \6_{\y^1}f\cdot \VV_d^1 +\6_{\y^1}^*f\cdot \6_{\y^2}f\cdot \VV_d^2=g\cdot \VV_d^1+\6_{\y^1}^*f\cdot \6_{\y^2}f\cdot \VV_d^2.$$
Therefore the membership $f(\zz)\in \6_{\y}f(\zz)\cdot \VV_d$ is equivalent to saying that the remainder of the componentswise division of $\6_{\y^1}^*f(\zz)\cdot f(\zz)$ by $g(\zz)\cdot \VV_d^1$ belongs to the image of $\6_{\y^1}^*f(\zz)\cdot \6_{\y^2}f(\zz)\cdot \VV_d^2$ in the factor module $\AAA^\kk/g(\zz)\cdot \VV_d^1$. But this quotient is a finite dimensional $\kkk$-vector space, so the membership defines by Theorem \ref{division} (b) a Zariski closed subset $\ZZ_d^*$ of $\ZZ_d$.
\end{proof}


\begin{rmks} (a) It is in general not true that the fibrations $\tau_d:\WW_d\map\ZZ_d^*$ and $\pi_d:\YY_d\map\ZZ_d^*$ are trivial or \ttextile locally trivial. In fact, the dimension of the fibers $\WW_{d,\zz}$ may vary with $\zz$, since the equation $f(\zz)+\6_yf(\zz)\cdot v=0$ defining $\WW_{d,\zz}$ has to be solved for $v=v(\t)$ inside $\AAA^m$ (recall here that the dependence of the series on the variable $\t$ is not marked notationally). This is only possible for those $\zz$ for which $f(\zz)$ belongs to the $\AAA$-module generated by the vectors $\6_{\y_i}f(\zz)$.

The statement of the fibration theorem holds under the more general assumption that $f$ admits an arbitrary division module $\GG$ (and not just one of the form described in Proposition \ref{divisionmodule} and remark \ref{spaces}(a)). The sets $\SS_\dd$, $\VV_\dd$, $\RR_\dd$ and $\ZZ_\dd$, with $\dd=(d_1,...,d_m)\in\N^m$, have then to be defined as in the linearization theorem \ref{generallinearization}.
In the factorization theorem \ref{factorization} below it will be shown that there exists a choice of a division module $\GG$ of $f$ such that the fibrations $\tau_d$ and $\pi_d$ actually become trivial.\medskip


(b) The quasi-affine variety $\ZZ_d^*$ appears implicitly in the proof of the approximation theorem \cite{Ar1, Pl}: it corresponds to the set of approximate solutions of the equation $f(\y)=0$ modulo the square of the minor $g$.\medskip


(c)  The theorem gives no statement about the complement $\YY_\infty=\YY_\infty(f)=\YY(f)\sm \bigcup_{d\in\N} \YY_d(f)$ of power series vectors $y$ where the minor $g$ vanishes. To describe this set one will have to apply induction on the height of the ideal $I$ generated by $f_1,...,f_\kk$, see Theorem \ref{partition} for the details.\medskip


(d) In the situation of arc spaces (i.e., when $f$ does not depend on $\t$), the sets $\YY_d$ are also known as the {\it contact locus}, cf. \cite{dFEI}. 
\end{rmks}


\section{Factorization theorem}

In this and the next section we will prove that every subvariety 
$$\YY=\YY(I)=\{y\in\AAA^m,\, f(y)=0 \text{ for all } f\in I\}$$
defined by some ideal $I$ of $\BB$ admits a countable stratification into \ttextile locally closed strata which are textile isomorphic to cartesian products of finite dimensional varieties with finite free $\AAA$-modules as stated in Theorem \ref{structuretheorem} from the introduction. Let $f_1,...,f_\kk$ be generators of $I$, and set $f=(f_1,...,f_\kk)\in\BB^\kk$.\medskip

The fibration theorem \ref{fibration} provides in the case $\kk\leq m$ and for a chosen $(\kk\times \kk)$-minor $g$ of $\6_\y f$ a stratification of $\YY(f)\sm \YY(g)$ by strata $\YY_d(f)$ which are linearly fibered. In general, these are not trivial fibrations. We will show in this section that for a smart choice of the division module $\GG$ of $f$ the fibrations will indeed be trivial, i.e., each $\YY_d(f)$ is textile isomorphic to a cartesian product. This choice of $\GG$ has been used for instance by P\l oski for proving his parametrization theorem \cite{Pl}, and also by many other authors, see e.g.~\cite{DL, GK, Dr}. The precise statement is as follows.


\begin{theorem} [Cartesian factorization of \arquile varieties] \label{factorization} Let $I\subset \BB$ be an ideal generated by series $f_1,...,f_\kk$ in $\t$ and $\y$ for which there exists a $(\kk\times \kk)$-minor $g$ of the relative Jacobian matrix $\6_{\y}f$ of $f=(f_1,...,f_\kk)$ which does not vanish identically. Let 
$$\YY=\YY(I)=\{y\in\AAA^m,\, f(y)=0\}\subset\AAA^m$$
be the \arquile variety defined by $I$ in $\AAA^m$. Write $\YY$ as the disjoint union 
$$\YY=\bigcup_{d\in\N\cup\{\infty\}} \YY_d,$$
of \ttextile locally closed subsets, where
$$\YY_d=\{y\in\YY,\, g(y)\neq 0,\, {\rm ord}\, g(y)=d\}$$
for $d\in\N$, and where
$$\YY_\infty=\{y\in\YY,\, g(y)=0\}.$$
%

Then, for $d\in\N$, the strata $\YY_d$ are textile isomorphic (and hence also $\t$-adically homeomorphic) over a Zariski locally closed subset $\ZZ_d^*$ of $\A_\kkk^{d(2m-\kk)}$ to the cartesian product of $\ZZ_d^*$ with a free $\AAA$-module of rank $m-\kk$,
%
%


$$\xymatrix{ \YY_d \ar[rr]^{\Phi_d} \ar[rdd]_{\pi_d} & & \ZZ_d^*\times \AAA^{m-\kk} \ar[ldd]\\
& & \\
 &\ZZ_d^*  & }$$

\end{theorem}


\begin{rmks}\label{rmk_factorization}  (a) This result specifies the assertions of the structure theorem \ref{structuretheorem} from the introduction.\medskip


(b) For $d\in \N$, the strata $\YY_d$ are \ttextile locally closed and $\t$-adically open in $\AAA^m$, and \ttextile cofinite locally closed in $\YY$.\medskip


(c) The stratum $\YY_\infty$ is \arquile closed in $\YY$ and defined in $\AAA^m$ by the ideal $I+\langle g\rangle$. If $I$ is prime of height $r$ and $g\not \in I$, this ideal has height $r+1$ and can therefore be submitted to induction to find its respective stratification, see Theorem \ref{partition}. \medskip


(d) The Zariski locally closed subset $\ZZ_d^*$ of $\A_\kkk^{d(2m-\kk)}$ is constructed as follows. We may assume that $g$ is the $(\kk\times \kk)$-minor defined by the first $\kk$ columns of $\6_\y f$. Consider the $\BBB$-submodule 
$$\GG=g\cdot\BBB^\kk\times g^2\cdot\BBB^{m-\kk}$$
of $\BBB^m$. Proposition \ref{divisionmodule} ensures that $\GG$ is a division module for $f$, cf. definition \ref{divisionmoduledef}. For $d\in\N$ set 
$$\SS_d=\{y\in \AAA^m,\, g(y)\neq 0,\, \text { ord}\, g(y)=d\}.$$
The evaluation $\GG(y)$ does not depend on the choice of $y\in\SS_d$ and equals
$$\VV_d=\GG(y)=\t^d\cdot\AAA^\kk\times \t^{2d}\cdot\AAA^{m-\kk}\subset \AAA^m.$$
We let $\RR_d\isom\A_\kkk^{d(2m-\kk)}$ denote the direct complement of $\VV_d$ as in Theorem \ref{linearization}, consisting of vectors in $\AAA^m$ whose first $\kk$ components are polynomials of degree $\leq d$ and whose last $m-\kk$ components are polynomials of degree $\leq 2d$. Set $\ZZ_d=\SS_d\cap\RR_d$ and define
$$\ZZ_d^*=\{z\in\ZZ_d,\,f(z)\in \6_{\y}f(z)\cdot \VV_d\}.$$
By assertion (c) of Theorem \ref{fibration}, this set is Zariski locally closed in affine space $\A_\kkk^{d(2m-\kk)}$. \medskip


(e) A suitable isomorphism $\Phi_d:\YY_d\map \ZZ_d^*\times\AAA^{m-\kk}$ is given by the restriction to $\YY_d$ of the composition $\lambda_d\circ \phi_d\circ\psi_d$ of the map
$$\phi_d\circ\psi_d:\SS_d\map \ZZ_d\times \VV_d,$$
$$y=\vv+\zz=\DD(\zz)\cdot a+\zz\map (\zz,\DD_g(\zz)\cdot (a+p(\zz,a))),$$
defined by the diagonal matrix $\DD=\diag(g,...,g,g^2,...,g^2)$ and the power series vector $p$ as in Proposition \ref{isomorphisms} and Theorem \ref{linearization} with the map
$$\lambda_d: \ZZ_d\times\VV_d\map \ZZ_d\times \AAA^\kk\times\AAA^{m-\kk},$$
$$(\zz,\vv)=(g(\zz)\cdot a_1, g(\zz)^2\cdot a_2,\zz)\map (\zz, a_1-\6^*_{\y^1}f(\zz)\cdot \6_{\y^2}f(\zz) \cdot a_2,a_2),$$
where $a=(a_1,a_2)$ and $\y=(\y^1,\y^2)$ denote the decompositions of $a$ and $\y$ into the first $\kk$ and last $m-\kk$ components and where $\6^*_{\y^1}f$ is the adjoint matrix of $\6_{\y^1}f$. \medskip


(f) As the minor $g$ varies, the sets $\YY_d$ will cover the entire \arqregular locus of $\YY$, see the next section.
\end{rmks}


\begin{proof} All constructions and arguments below will be compatible with the restrictions to the convergent or algebraic power series setting. The set $\YY_d$ equals the intersection $\YY\cap \SS_d$. In Theorem \ref{fibration} it was shown that $ \YY_d$ is textile isomorphic over $\ZZ_d^*$ to
$$\WW_d=\{(\zz,\vv)\in \ZZ_d\times\VV_d,\, f(\zz)+ \6_{\y}f(\zz)\cdot \vv=0\}$$
over $\ZZ_d^*$. The projection $\tau_d: \WW_d\map \ZZ_d^*,\, (\zz,\vv)\map\zz$ is by Theorem \ref{fibration} a fibration with fibers $\WW_{d,\zz}$ which are affine $\AAA$-modules. We will show that this fibration is trivial for our specific choice 
$$\GG=g \cdot\BBB^\kk\times g^2\cdot\BBB^{m-\kk}$$
of the division module $\GG$. To this end we will construct a textile isomorphism $\WW_d\map\ZZ_d^*\times \AAA^{m-\kk}$ over $\ZZ_d^*$.\medskip


Set $\y^1=(\y_1,\ldots,\y_\kk)$ and $\y^2=(\y_{\kk+1},\ldots,\y_m)$. For $\zz\in\ZZ_d$, the evaluation $g(\zz)$ of the minor $g$ of the submatrix $\6_{\y^1}f(\zz)$ is non-zero, and therefore the linear map $\AAA^\kk\map \AAA^\kk$ induced by the adjoint matrix $\6^*_{\y^1}f$ of $\6_{\y^1}f$ is injective. By multiplying the system of equations $f(\zz)+\6_{\y}f(\zz)\cdot \vv=0$ from the left with $\6^*_{\y^1}f$ we hence obtain the equivalent system of equations 
$$\6^*_{\y^1}f(\zz)\cdot f(\zz)+\6^*_{\y^1}f(\zz)\cdot \6_{\y}f(\zz)\cdot v=0.$$
For $\zz\in\ZZ_d^*$ and $\vv\in\VV_d=\GG(\zz)$, write $v=(g(\zz)\cdot a_1,g(\zz)^2\cdot a_2)$ with $a_1\in\AAA^\kk$ and $a_2\in\AAA^{m-\kk}$. As $\6^*_{\y^1}f\cdot \6_{\y^1}f=g\cdot\1_\kk$,  the previous system can be rewritten as
$$\6^*_{\y^1}f(\zz)\cdot f(\zz)+g(\zz)^2 \cdot a_1+\6^*_{\y^1}f(\zz)\cdot \6_{\y^2}f(\zz)\cdot g(\zz)^2 \cdot a_2=0,$$
say,
$$\6^*_{\y^1}f(\zz)\cdot f(\zz)+ g(\zz)^2 \cdot [a_1+\6^*_{\y^1}f(\zz)\cdot \6_{\y^2}f(\zz) \cdot a_2]=0.$$
This shows that $\ZZ_d^*$ can equally be defined as
$$\ZZ_d^*=\{z\in\ZZ_d,\, \6^*_{\y^1}f(z)\cdot f(z)\in\langle g(\zz)^2\rangle\cdot\AAA^\kk\}.$$
%
%
Consider now for each $\zz\in\ZZ_d^*$ the isomorphism
$$\lambda_{d,\zz}: \VV_d\map \AAA^\kk\times\AAA^{m-\kk},$$
$$\vv=(g(\zz)\cdot a_1,g(\zz)^2\cdot a_2)\map (a_1-\6^*_{\y^1}f(\zz)\cdot \6_{\y^2}f(\zz) \cdot a_2,  a_2).$$
The preceding computations show that the fibers $\WW_{d,\zz}$ are sent by $\lambda_{d,\zz}$ to the solutions $a=(a_1,a_2)\in\AAA^r \times \AAA^{m-\kk}$ of the system of equations
$$\6^*_{\y^1}f(\zz)\cdot f(\zz)+ g(\zz)\cdot g(\zz) \cdot a_1=0.$$
As $\6^*_{\y^1}f(z)\cdot f(z)\in g(\zz)\cdot g(\zz)\cdot\AAA^\kk$ for $\zz\in\ZZ_d^*$, these solutions are of the form 
$$(a_1,a_2)=(-g(\zz)^{-2}\cdot \6^*_{\y^1}f(\zz)\cdot f(\zz),a_2),$$
with arbitrary $a_2\in\AAA^{m-\kk}$. It follows that the map
$$\lambda_d: \WW_d\map \ZZ_d^*\times\AAA^{m-\kk},\, (\zz,\vv)\map(\zz,\lambda_{d,\zz}(\vv))$$
is a textile isomorphism over $\ZZ_d^*$. Composing $\lambda_d$ with $\phi_d\circ\psi_d:\SS_d\map \ZZ_d\map \VV_d$ as in Proposition \ref{isomorphisms} gives the required isomorphism $\Phi_d=\lambda_d\circ \phi_d\circ\psi_d$ for $\YY_d$. This concludes the proof of the factorization theorem.
\end{proof}


\section{Partition of \arquile varieties} \label{partitions}

In this section we indicate how to decompose and stratify an arbitrary \arquile subvariety $\YY(I)$ of $\AAA^m$ such that on each stratum the factorization theorem \ref{factorization} can be applied. This, in turn, will then establish the structure theorem \ref{structuretheorem} for \arquile varieties from the introduction.\medskip

So let $I\subset \BB$ be an ideal, and let $\YY=\YY(I)$ be its associated zeroset in $\AAA^m$. Without loss of generality we may assume that $I$ is saturated, i.e., that $I$ equals the ideal $I_{\YY(I)}$ as defined in section 2.  In particular, $I$ will be radical.  Let
$$I=I_1\cap\ldots\cap I_s$$
be the irredundant prime decomposition of $I$, with $I_j\subset \BB$ prime. By Proposition \ref{primedecomposition}, all ideals $I_j$ are again saturated. We get
$$\YY(I)=\YY(I_1)\cup\ldots\cup \YY(I_s),$$
and each $\YY(I_j)$ is \arquile closed in $\YY$. 
Recall that $\Reg(\YY)$ denotes the \arqregular locus of points $y$ of $\YY$ for which $(\BB/I)_{\PP_y}$ is a regular local ring, where $\PP_y\subset\BB$ is the prime ideal of relations among the components of $y$. This locus is \arquile open in $\YY$. For $y\in\Reg(\YY)$, let $r_y$ be the height of the ideal $I\cdot \BB_{\PP_y}$. For $r\in\N$, denote by $\YY_r$ the \arquile locally closed subset of $\YY$ of points $y\in \Reg(\YY)$ with $r_y=r$. By part (c) of Theorem \ref{singularlocus}, each $\YY_r$ is covered by finitely many \arquile open subsets of the form $\YY(f)\sm\YY(g)$, where $f=(f_1,...,f_r)$ is a vector of elements in $I$ and $g$ is a suitable $(r\times r)$-minor of the relative Jacobian matrix $\6_\y f$ of $f$. Varying $r$ then yields \an \arquile and $\t$-adically open covering of $\Reg(\YY)$ by such sets. \medskip

By part (e) of Theorem \ref{singularlocus}, the \arqsingular locus $\Sing(\YY)$ is an \arquile closed proper subset of $\YY$. As $\AAA^m$ is Noetherian with respect to the \arquile topology, we may apply induction to stratify $\Sing(\YY)$ suitably, starting again with its \arqregular locus $\Reg(\Sing(\YY))$ and then passing on to its \arqsingular locus $\Sing(\Sing(\YY))$. Note that the strata will now be just \arquile locally closed in $\YY$. Putting all this together then gives




\begin{theorem} [Partition of \arquile varieties] \label{partition} Every \arquile subvariety $\YY=\YY(I)$  of $\AAA^m$ admits a finite stratification $\YY=\bigsqcup\, \YY_i$ into \arquile locally closed subsets $\YY_i=\Sing_i(\YY)\sm\Sing_{i+1}(\YY)$ such that each $\YY_i$ has a finite \arquile and $\t$-adically open covering by sets of the form $\UU_g=\YY(f)\sm\YY(g)$, where $f\in\BB^\kk$ is a vector of power series and where $g$ is a minor of the relative Jacobian matrix $\6_\y f$ of $f$ of size equal to the height of the ideal of $\BB$ generated by the components of $f$.
\end{theorem}


\section{Example}

\begin{example} Let $\YY(f)\subset \AAA^3$ be the arquile variety given by $f(\y) = \y_{1}^{2} - \y_{2} \y_{3}$. Since the parameter $\t$ does not appear in the equation and since $f$ is irreducible, the singular locus equals $\Sing(\YY) = \YY(f, \partial_{\y_{1}}f,\partial_{\y_{2}}f,\partial_{\y_{3}}f)  =\{(0,0,0)\}$ and the \arqregular part is covered by $\YY(f) \setminus \YY(\partial_{\y_{i}}f)$. 
We fix $i=1$, so we consider the set $\YY(f)\setminus \YY(\y_{1})$, which we stratify into the sets $(d\geq 1)$
$$\YY_{d} = \{y \in \YY(f),\, \ord\,  y_{1} = d\}.$$

It is easily seen that 
$$\ZZ_{d} = \{ (z_{1,d} \t^{d}, \sum_{k=1}^{2d} z_{2,k} \t^{k}, \sum_{k=1}^{2d} z_{3,k} \t^{k}),\, z_{1,d}\ \neq 0 \}.$$
In the hypersurface case, the defining minor for $\ZZ_{d}$ stems from the derivative $\partial_{\y_{1}}f$, whence $\partial^{*}_{\y_{1}} f = 1$. Therefore the equation for $\ZZ_{d}^{*}$ simplifies to $f(\zz) \equiv 0 \mod \t^{2d + 1}$, which is equivalent to the system 
%
\begin{align*} z_{1,d} &\neq  0,\\
 z_{2,1}z_{3,1} &=0,\\
  &\vdots \\
  \sum_{j=1}^{2d-1} z_{2,2d-1 - j} z_{3,j} &= 0,\\
 z_{1,d}^{2} - \sum_{j=1}^{2d} z_{2,2d - j} z_{3,j} &= 0.
\end{align*}

By the cartesian factorization theorem \ref{factorization}, $\YY_{d}$ is isomorphic to 
$\ZZ^{*}_{d} \times \AAA^{2}$. An isomorphism 
$$\Phi^{-1}_{d} = (\Phi^{-1}_{d,1},\Phi^{-1}_{d,2},\Phi^{-1}_{d,3}) \colon  \ZZ^{*}_{d} \times \AAA^{2} \rightarrow \YY_{d}$$
can be explicitly computed and is given by
%
$$\Phi_{d,1}^{-1}(z_{1},z_{2},z_{3}, a_{2},a_{3}) = z_{1} + \t^{d}\cdot(-z_{1,d} + \sqrt{\Xi(z_{1},z_{2},z_{3}, a_{2},a_{3}) }),$$
$$\Phi_{d,2}^{-1}(z_{1},z_{2},z_{3}, a_{2},a_{3}) = z_{2} + \t^{2d} a_{2},$$
$$\Phi_{d,3}^{-1}(z_{1},z_{2},z_{3}, a_{2},a_{3}) = z_{3} + \t^{2d}a_{3},$$
where
$$\Xi(z_{1},z_{2},z_{3}, a_{2},a_{3}) =  z_{1,d}^{2} - \frac{z_{1}^2 - z_{2}z_{3}}{\t^{2d}} +z_{2}a_{3} + z_{3}a_{2} + \t^{2d} a_{2} a_{3}.$$
Here is a short explanation for these formulas. We expand $f(z + v) $ into $ f(z) + \partial_{\y}(f)(z) \cdot v + q(z,v)$, say
$$ f(z + v) = (z_{1}^{2}  -z_{2}z_{3}) + 2z_{1}v_{1} - z_{2}v_{3} - z_{3}v_{2} + (v_{1}^{2} - v_{2} v_{3} ).$$
Substitute $v_{1} = \t^{d} a_{1}, \ v_{2} = \t^{2d}a_{2} ,\ v_{3} = \t^{2d} a_{3}$. We assume that $z \in \mathcal{Z}_{d}^{*}$,
so there exists a series $h_{z}(\t) \in \AAA$ such that $z_{1}^{2} - z_{2}z_{3} = \t^{2d} \cdot h_{z}(t)$. 
We want to solve $f(z + v) = 0$, which in ${\rm a}$-coordinates takes the form
$$(z_{1}^{2}  -z_{2}z_{3}) + 2z_{1}\t^{d}a_{1} - z_{2}\t^{2d}a_{3} - z_{3}\t^{2d}a_{2} + (\t^{2d}a_{1}^{2} - \t^{2d}a_{2} \t^{2d}a_{3} )=0,$$
say
$$\t^{2d} \cdot\left( h_{z}(\t) + 2z_{1,d}a_{1} - z_{2}a_{3} - z_{3}a_{2} + a_{1}^{2} - \t^{2d}a_{2} a_{3}  \right) = 0.$$
Therefore
$$ a_{1} ={} -z_{1,d} + \sqrt{ z_{1,d}^{2} - \frac{z_{1}^2 - z_{2}z_{3}}{\t^{2d}} + z_{2} a_{3} + z_{3}a_{2} + \t^{2d} a_{2}a_{3}}$$
as claimed (the square root exists since the order in $t$ of the series is $0$).
\end{example}

\goodbreak


\goodbreak

\section{Deformations}

In this section we formulate the analogue of the factorization theorem \ref{factorization} for the set $\YY(f)_S$ of deformations $\wt y(\t)$ of a given power series vector $y(\t)\in \YY(f)\subset\AAA^m$ over a base ring $S$. Considering deformations of $y$ corresponds to working, in a certain sense, locally at $y$. To establish the cartesian product structure of the set of deformations it will no longer be necessary to stratify the \arquile variety $\YY(f)$, and the partition theorem \ref{partition} becomes redundant: one can directly prove the factorization. Taking in particular deformations parametrized by the elements of a test algebra $S$ (a local ring whose maximal ideal is nilpotent), the results of this section contain the factorization of formal neighborhoods of arc spaces as proven by Grinberg-Kazhdan and Drinfeld \cite{GK, Dr, BH, BS1, BS2, Bou}.\medskip

We restrict to the case of formal power series and let $\AAA=\k[[\t]]$ always denote the ring of series without constant term. Similar results as below can be proven in the convergent setting, but the details are more involved and require Banach-space techniques as developed in \cite{HM}.\medskip

Let $S$ be a local, not necessarily Noetherian ring, with maximal ideal $\mm_S$ and residue field $S/\mm_S=\kkk$ embedding into $S$ as a subfield. We shall always assume that $S$ is {\it complete} with respect to a filtered topology defined by a decreasing sequence of ideals $J_k$ with $J_k\cdot J_\ell\subset J_{k+\ell}$ and such that $\mm_S^k\subset J_k$.\medskip


Denote as in the section on division by $\AA_{S,\circ}=S_\circ[[\t]]$ the ring of formal power series in $\t$ with coefficients in $S$ and zero constant term.
%
%
We shall write $\wt y(\t) =(\wt y_1(\t),...,\wt y_m(\t))$ for power series vectors in $\AA_{S,\circ}^m$, with coefficient vectors $(\wt y_{1j},...,\wt y_{mj})$ in $S^m$ of the monomials $\t^j$, for $j\geq 1$. Their residue classes $y(\t)$ modulo $\mm_S$ are power series vectors in $\AAA^m$, and $\wt y(\t)$ is then called a {\it deformation} of $y(\t)$. We may thus write 
$$\wt y(\t)= y(\t)+\wt y\,'(\t)$$
where $\wt y\,'(\t)\in\mm_S[[\t]]^m$ is a vector with coefficients in $\mm_S$, say, a deformation of $0\in\kkk^m$. The set of all deformations $\wt y(\t)$ of a vector $y=y(\t)\in\AAA^m$ parametrized by $S$ is denoted by $(\AA_{S,\circ}^m,y)$, say
$$ (\AA_{S,\circ}^m,y) = y(\t)+ \mm_S\cdot \kkk[[\t]]^m.$$
Let $\s_{ij}$ denote countably many variables, for $1\leq i \leq m$ and $j\geq 1$, and let $\kkk[[\s_{m,\infty}]]$ be the power series ring $\kkk[[\s_{ij},\, 1\leq i \leq m,\, j\geq 1]]$ in these variables, with maximal ideal $(\s_{ij},\, 1\leq i \leq m,\, j\geq 1)$. For a given vector $y(\t)\in \AAA^m$, the vector $\wt y(\t)=(\wt y_1(\t),...,\wt y_m(\t))$ given by
$$ \wt y_i(\t)= y_i(\t)+\sum_{j\geq 1} \s_{ij}\cdot \t^j$$
defines a {\it universal} deformation of $y(\t)$, cf.~\cite{Ha1}: any other deformation of $y(\t)$ over a ring $S$ is obtained from this one by the evaluation of $\s_{ij}$ at elements of $M_S$, i.e., by base change. \medskip

Let $\BB$ be defined as in the earlier sections as the space of formal, convergent or algebraic power series in $\t$ and variables $\y_1,...,\y_m$. Vectors $f(\t,\y)\in\BB^\kk$ with $f(0)=0$ and with associated arquile map $f_\infty:\AAA^m\map\AAA^\kk$ define, for every ring $S$, also a map on $\AA_{S,\circ}^m$, 
$$f_{S,\infty}:\AA_{S,\circ}^m\map \AA_{S,\circ}^\kk,\, \wt y(\t)\map f(\t,\wt y(\t)),$$
called the {\it \arquile map over $S$ induced by $f$}, and a map
$$(f_{S,\infty},y):(\AA_{S,\circ}^m,y)\map (\AA_{S,\circ}^\kk, f_\infty(y)),\,\wt y(\t)\map f(\t,\wt y(\t))$$
between the respective sets of deformations of $y\in\AAA^m$ and $f_\infty(y)\in\AAA^\kk$. For \an \arquile subvariety $\YY=\YY(f)$ of $\AAA^m$ defined by a vector $f\in\BB^\kk$, we shall write 
$$(\YY_S,y)=(\YY_S(f),y)$$
for the set of deformations $\wt y$ of $y$ ``lying in $\YY$'', i.e., satisfying $f_{S,\infty}(\wt y)=0$. Expanding $f_{S,\infty}(\wt y)=f(\t,\wt y(\t))$ as a power series in $\t$ one sees that these deformations are defined by setting the coefficients of $\t^j$ in $f(\t,\wt y(\t))$ equal to $0$ for all $j\geq 0$. This then induces polynomial equations for the coefficients of $\wt y$ because the constant term of $\wt y$ is assumed to be zero. Taking in particular $S=\kkk[[\s_{m,\infty}]]$, we get in this way an ideal $I(\YY,y)_\infty$ in $S$ of {\it universal} polynomials, respectively power series, defining the deformations of $y$ in $\YY(f)$ over $\kkk[[\s_{m,\infty}]]$. We may then identify deformations in $(\YY_S(f),y)$ with deformations of $y$ in $\AA_{S'}^m$ parametrized by the factor ring $S'=S/I$, with $I=I(\YY,y)_\infty$.\medskip


\begin{example} Let $y(\t)=(\t^2,\t^3)$ be the canonical solution of $f(\y)=\y_1^3-\y_2^2=0$, and consider a deformation $\wt y(\t)=(\t^2+\s_1\t,\t^3+\s_2\t^2)$ of $y(\t)$ with $S=\kkk[[\s_1,\s_2]]$. Then $f(\wt y(\t))=0$ if and only if 
$$3\t^4\s_1\t+3\t^2\s_1^2\t^2+\s_1^3\t^3=2\t^3\s_2\t^2+\s_2^2\t^4,$$
say,
$$3\s_1=2\s_2, \, 3\s_1^2=\s_2^2,\, \s_1^3=0.$$
In this case, the ideal $I(\YY,y)_\infty$ is not radical.
\end{example}


If $Z$ is an algebraic variety or a scheme of finite type over $\kkk$, and $z\in Z$ a $\kkk$-point, we may identify the set $(Z_S,\zz)$ of deformations of $\zz$ in $Z$ with the set of morphisms $\wt \zz:\Spec(S)\map Z$ sending the closed point of $\Spec(S)$ to $\zz$. \medskip

For the simplicity of the exposition we will restrict in this section to {\it polynomial} vectors $f\in\kkk[\y]^\kk$ which, moreover, do not depend on $\t$, so that the \arquile variety $\YY(f)$ coincides with the space $X_{\infty,0}$ of arcs centered at $0$ of the algebraic variety $X$ defined in $\A^m_\kkk$ by $f=0$. This restriction is, however, not a substantial assumption, and all statements can be extended conveniently to the general case. The main result of this section is a reinterpretation of Theorem \ref{cohen} of the introduction, replacing the formal neighborhood of an arc $y$ by its set of deformations.


\begin{theorem} [Structure theorem for deformations]\label{deformations}
Let $X\subset{\mathbb A}_\kkk^m$ be an algebraic variety defined by a polynomial vector $f$ in $\kkk[\y]^\kk$, and let $\YY=\YY(f)\subset\AAA^m$ be the associated \arquile variety (viz the arc space $X_{\infty,0}$ of $X$ of arcs centered at $0$). Let $y=y(\t)\in \Reg(\YY)$ be an arc in $\YY$ not lying entirely in the singular locus ${\rm Sing}\, X$ of $X$. There then exists a scheme $Z$ of finite type over $\kkk$, a point $\zz$ of $Z$, an integer $e$ and a point $b\in\AAA^e$, such that for all complete local rings $S$ one has bijections
$$\Phi_S:(\YY_S,y)\isom (Z_S,z)\times (\AA_{S,\circ}^e,b)$$
\end{theorem}


\begin{rmks} (a) The maps $\Phi_S$ are given by the composition of a division map with \an \arquile map over $S$ induced from \an \arquile map over $\kkk$. In this sense they are ``functorial'' with respect to $S$, see the proof.\medskip

(b) Taking for $S$ a $\kkk$-algebra with nilpotent maximal ideal $\mm_S$ one recovers the factorization theorem for arc spaces given by Grinberg-Kazhdan and Drinfeld \cite{GK, Dr, BH, Bou, BNS}.%
%
%
\footnote{ To be more precise, one would have to consider here also deformations with non-zero constant term. This can be done since the power series are substituted into polynomial equations.}
In \cite{BS2}, variations and extensions of this result have been proven.
\vskip .2cm

(c) The isomorphism $\Phi$ is {\it embedded} in the sense that it is given as the restriction of a suitable bijective map $\wh \Phi$ which is defined on the ambient space $(\AA_{S,\circ}^m,y)$ of $(\YY_S,y)$.\medskip

(d) The assertion of the theorem is not a direct consequence of the factorization theorem \ref{factorization}, since the deformations $\wt y$ in $(\YY_S,y)$ are not required to lie entirely in the stratum $\YY_d$ of $\YY$ in which the vector $y$ lies (as the parameter varies, the order of the minor $g$ may drop).\medskip

(e) In the case where $X\subset \wh \A^m_\kkk$ is a formal variety defined by a formal power series vector $f\in\kkk[[\y]]$, the same proof as for the theorem goes through with only notational modifications. If $X\subset (\C^m,0)$ is the germ of a complex analytic variety defined by a vector $f\in\C\{\y\}$, things are getting more complicated: the ring $S$ has then also to be a convergent power series ring for which one has to admit only analytic deformations. The structure of the proof is the same, but at each step one has to make sure that all constructed series converge suitably. One may consult \cite{Ha1} to get an impression for the flavour of the required techniques.\medskip

(f) In the next section we will describe an approach to cartesian product structures of deformation spaces via derivations and analytic triviality.
\end{rmks}


\begin{proof} The proof runs parallel to the proof of the factorization theorem  \ref{factorization} for arquile varities, so we only indicate the necessary modifications. The main difference is the use of the Weierstrass division theorem \ref{divisiondeformations} in the version for deformations instead of the classical version \ref{division}. It is here that we need to work with a {\it complete} ring $S$. The same line of arguments proves Theorem \ref{cohen}.\medskip

The first step is to reduce to a vector $f$ which has {\it full generic rank}, i.e., for which a maximal minor of the Jacobian matrix $\6_\y f$ does not vanish, cf. Theorem \ref{singularlocus}. This is quite standard: As the arc $y$ does not lie entirely in the singular locus of $X$, it must lie in precisely one irreducible component. By Lemma 8.6 of \cite{CdFD}, see also section 3.4 in \cite{CNS}, we may assume that $X$ is irreducible. Represent $X$ as an irreducible component of a complete intersection $X^*$. Using again that $y$ lies in precisely one component of $X^*$, namely $X$, we may assume from the beginning that $X$ is a complete intersection. This is equivalent to saying that $f$ has full generic rank. 
\medskip

Let $r$ be the height of the ideal $I$ of $\kkk[\y]$ generated by the components of $f$. The singular locus is defined by the vanishing of the $(r\times r)$-minors of $\6_\y f$. As $y$ does not lie in ${\rm Sing}\,X$, there exists an $(r\times r)$-minor $g$ of $\6_\y f$ which does not vanish upon the substitution of the variables $\y$ by the vector $y$.  So we have $g(y)\neq 0$. Let $d$ be the order of $g(y)$ as a power series in $\t$, and denote by $\DD_g$ the diagonal matrix 
$$\DD_g=\diag(g,...,g,g^2,...,g^2)$$
as in remark (d) following the factorization theorem \ref{factorization}. Decompose $y=\vv+\zz=a\cdot \DD_g(\zz)+\zz$ according to the classical division theorem \ref{division}, with $a,\zz\in \AAA^m$ and $\vv=a\cdot \DD_g(\zz)\in \t^d\cdot\AAA^\kk\times \t^{2d}\cdot\AAA^{m-\kk}$. Here, $\vv$ is a power series vector in $\t$ without constant term whose first $\kk$ components have order $\geq d+1$ and whose last $m-\kk$ components have order $\geq 2d+1$. Similarly, $\zz\in \k[\t]^\kk_{\leq d}\times \k[\t]^{m-\kk}_{\leq 2d}$ is a polynomial vector in $\t$ without constant term whose first $\kk$ components have degree $\leq d$ and whose last $m-\kk$ components have degree $\leq 2d$.\medskip

By the Weierstrass division theorem \ref{divisiondeformations} for deformations we may then decompose any deformation $\wt y$ of $y$ into 
$$\wt y =  \wt \vv+ \wt \zz =\wt a\cdot \DD_g(\wt \zz) +\wt \zz,$$
where $\wt \vv$, $\wt \zz$ and $\wt a$ are deformations of the vectors $\vv$, $\zz$ and $a$, and where $\wt \zz$ is a polynomial vector in $\t$ whose first $\kk$ components have degree $\leq d$, and whose last $m-\kk$ components have degree $\leq 2d$. The coefficients of $\wt \vv$, $\wt \zz$ and $\wt a$ are formal power series in the coefficients of $\wt y$.%
\footnote{ In contrast to the situation over a field $\kkk$, these coefficients can now be genuine series, see the example after the proof.}
Note that, by an argument as in remark \ref {rmkdivisionmodule} (a), the matrix $\DD_g(\wt\zz)$ equals $\DD_g(\wt y)$ up to the multiplication with an invertible matrix with entries in $S[[\t]]$: this holds because $\wt y$ and $\wt\zz$ differ from $y$ and $\zz$ by series with coefficients in the maximal ideal $\mm_S$ of $S$. \medskip

Set 
$$\RR_{d,S}=(\k[\t]_{\leq d}^\kk\times \k[\t]_{\leq 2d}^{m-\kk})_S=S_\circ[\t]_{\leq d}^\kk\times S_\circ[\t]_{\leq 2d}^{m-\kk},$$
so that $\RR_{d,S}=(\RR_d)_S$ for $\RR_d= \k[\t]_{\leq d}^\kk\times \k[\t]_{\leq 2d}^{m-\kk}$, where $S_\circ[\t]$ denotes the space of polynomials in $\t$ over $S$ whose constant term is zero. Denote by  $(\RR_{d,S},\zz)$ the set of deformations $\wt\zz$ of $\zz$ in $\RR_{d,S}$, with $\zz$ given by the division $y=a\cdot \DD_g(\zz)+\zz=\vv+\zz $.\medskip

Define bijective maps 
$$\psi_S:(\AA_{S,\circ}^m,y)\map (\RR_{d,S},\zz)\times(\AA_{S,\circ}^m,a), $$
$$ \wt y=\wt a\cdot \DD_g(\wt\zz)+\wt \zz\map (\wt \zz,\wt a),$$
and
$$\phi_S: (\RR_{d,S},\zz)\times(\AA_{S,\circ}^m,a)\map (\RR_{d,S},\zz)\times(\AA_{S,\circ}^m,a+p(\zz,a)),$$
$$ (\wt \zz,\wt a)\map  (\wt \zz, \DD_g(\wt \zz)\cdot (\wt a+p(\wt \zz,\wt a))),$$
as in Proposition \ref{isomorphisms}, with $p$ associated to $\DD_g$ as in definition \ref{divisionmodule}. Let  
$$\chi_S =(\phi_S\circ\psi_S)^{-1}:(\RR_{d,S},\zz)\times(\AA_{S,\circ}^m,a+p(\zz,a))\map (\AA_{S,\circ}^m,y)$$
denote the inverse of the composition of $\phi_S$ with $\psi_S$. As in the proof of the  lineari\-zation theorem \ref{linearization} it follows that the map
$$(f_{S,\infty},y)\circ \chi_S: (\RR_{d,S},\zz)\times(\AA_{S,\circ}^m,a+p(\zz,a))\map (\AA_{S,\circ}^\kk,f(\zz)+\6_\y f(\zz)\cdot\vv)$$
$$(\wt\zz,\wt a)\map f(\wt \zz)+\6_\y f(\wt \zz)\cdot \wt a\cdot \DD_g(\wt \zz)$$
is linear in $\wt a$. Define $(\ZZ^*_S,\zz)$ as the space of deformations $\wt \zz\in (\RR_{d,S},\zz)$ of $\zz$ such that $f(\wt \zz)\in \6_\y f(\wt \zz)\cdot (\AA_{S,\circ}^m,a)\cdot \DD_g(\wt\zz)$. Similarly as in the proof of the factorization theorem \ref{factorization} we may describe $(\ZZ^*_S,\zz)$ as the space of deformations $\wt \zz\in (\RR_{d,S},\zz)$ of $\zz$ such that 
$$\6^*_{\y^1}f(\wt\zz)\cdot f(\wt \zz)\in g(\wt \zz)^2\cdot \AA_{S,\circ}^\kk.$$
Recall here that $g(\zz)$ has order $d$ as a power series in $\t$, so that $g(\zz)^2$ has order $2d$. But $g(\wt\zz)^2$ may have smaller order in $\t$, stemming from monomials $\t^j$ for $j<2d$ with coefficients in $\mm_S$. Therefore, if we divide $\6^*_{\y^1}f(\wt\zz)\cdot f(\wt \zz)$ by $g(\wt\zz)^2$ according to Theorem \ref{divisiondeformations} with initial monomial $\t^{2d}$ and quotient in $\AA_{S,\circ}$, the remainder will be a polynomial vector in $\t$ of degree  $\leq 2d$ whose coefficients are only {\it formal power series} but not necessarily {\it polynomials} in the coefficients of the expansion of $\wt \zz$, see the example below. Equating these series to $0$ will give conditions on the coefficients of $\wt\zz$ which are equivalent to the membership $\6^*_{\y^1}f(\wt\zz)\cdot f(\wt \zz)\in g(\wt \zz)^2\cdot \AA_{S,\circ}^\kk$.\medskip

This contrasts the description of the set $\ZZ^*_d$ from Theorem \ref{fibration} (c), which was shown to be a quasi-affine variety. In order to remedy this drawback in the setting of deformations we use a trick from \cite{Dr}. Map $\wt \zz$ to the pair $(\wt\zz,\wt u)$, where $\wt u=(\wt u_0,...,\wt u_{2d-1})\in \AA_{S,\circ}^{2d}$ is the coefficient vector of the Weierstrass polynomial $h(\wt u)=\t^{2d} +\sum_{i=0}^{2d-1} \wt u_i\cdot \t^i$ of $g(\wt\zz)^2$. By Theorem \ref{divisiondeformations}, the components $\wt u_i$ are formal power series in the coefficients of $\wt\zz$. Obviously, the membership $\6^*_{\y^1}f(\wt\zz)\cdot f(\wt \zz)\in g(\wt \zz)^2\cdot \AA_{S,\circ}^\kk$ for vectors $\wt\zz$ is equivalent to the membership
$$\6^*_{\y^1}f(\wt\zz)\cdot f(\wt \zz)\in h(\wt u)\cdot \AA_{S,\circ}^\kk$$
of pairs $(\wt\zz,\wt u)$. But now we may divide $\6^*_{\y^1}f(\wt\zz)\cdot f(\wt \zz)$ {\it polynomially} by $h(\wt u)$, thus producing a remainder whose coefficients are polynomials in the coefficients of $\wt \zz$ and $\wt u$. Observe here that all constructions are ``functorial'' in $S$, i.e., do not depend on the special choice of $S$. The next step is to equate these coefficient polynomials to $0$ to get suitable equations for the required set of pairs $(\wt\zz,\wt u)$.  \medskip

Finally, define $b\in \AAA^{m-r}$ as the image of the vector $a$ under the map $\lambda_d\circ \phi_d\circ \psi_d$ from the proof of the  factorization theorem \ref{factorization}. Then proceed as in this proof to establish the required factorization
$$\Phi:(\YY_S,y)\isom (\ZZ_S,\zz)\times (\AA_{S,\circ}^{m-r},b).$$
But $\RR_d$ is a finite dimensional $\kkk$-vector space, so that $\RR_{d,S}$ can be identified with the set $(\RR_d)_S$ of morphisms $\Spec(S)\map \RR_d$. Therefore there exists, with the notation from the beginning of this section, a scheme $Z$ of finite type over $\kkk$ such that $(\ZZ_S,\zz)$ equals the space $(Z_S,\zz)$ of deformations of $\zz$ in $Z$ over $S$. 
\end{proof}


\begin{example} We illustrate the fact that the division of a deformation may produce a quotient and remainder whose coefficients are genuine power series in the elements of the maximal ideal $M_S$ of $S$. Let $S=\kkk[[\s_1,\s_2,\s_3,\s_4]]$ be the formal power series ring in four variables, and consider the polynomials $F=\s_1\cdot \t^5 +\s_2\cdot \t^3$, and $G=\t^4-\s_3\cdot \t^3 -\s_4\cdot \t^5$. Here, we treat $G$ as a {\it deformation} of $\t^4$ with parameters in the maximal ideal $\mm_S$ of the ring $S$. We describe the membership condition $F\in\langle G\rangle\subset S[[\t]]$ in terms of equations for the variables $\s_i$. The {\it virtual} Weierstrass form of $G$ is the polynomial $H=\t^4-\u\cdot \t^3$, for some new variable $\u$, since no smaller degree terms appear in the remainder of the division of $\t^4$ through $G$. Here, the series $u(\s_3,\s_4)$ yielding the actual Weierstrass form $h= \t^4-u(\s_3,\s_4)\cdot \t^3$ of $G$ is given by the implicit equation
$$\u^2\cdot \s_4-\u +\s_3 =0,$$
as can be seen by the polynomial division of $G$ through $H$. As $u$ is supposed to be a power series in $\s_3,\s_4$, the unique relevant solution of this equation is 
$$u(\s_3,\s_4)={1\over 2\s_4}\cdot (1-\sqrt{1+4\s_3\s_4}).$$
The other solution to the quadratic equations is ${1\over 2\s_4}\cdot (1+\sqrt{1+4\s_3\s_4})$. Its denominator $2\s_3$ does not cancel with a factor of the numerator and therefore this solution is not a power series in $\s_3,\s_4$ and can be discarded. The membership $F\in \langle G\rangle=G\cdot S[[\t]]$ is then equivalent to the polynomial equation
$$\s_1+\s_2\cdot \u^2=0,$$
which is just the (unique non-zero) coefficient of the remainder of the division of $F$ by $H$. But as soon as we eliminate the variable $u$ from this equation, by substituting it by the series $u(\s_3,\s_4)$, the equation is no longer polynomial in $\s_1,\s_2,\s_3,\s_4$. 
\end{example}


\section{Triviality}

In this section, we exhibit in an example why the formal neighborhood of an arc $y(\t)$ contained in the \arqsingular locus $\Sing(X_{\infty,0})$ of an arc space $X_{\infty,0}$ of an algebraic variety $X\subset\A^m_\kkk$ is not expected to satisfy the Grinberg-Kazhdan-Drinfeld theorem \ref{cohen}. This relates to similar considerations made by Chiu, de Fernex and Docampo in \cite{CdFD}. Our approach should be suited to extend the example of Bourqui and Sebag \cite{BS1} to more general situations. They show by means of differential algebras that the formal neighborhood of the constant arc $0$ in the arc space $X_{\infty,0}$ of the plane curve $X$ defined by $y_1^2+y_2^2=0$ in $\A^2_\kkk$ is not a cartesian product as in Theorem \ref{deformations}, provided that the field $\kkk$ does not contain a square root of $-1$. \medskip

It seems that the triviality techniques from local complex analytic geometry as developed by Ephraim \cite{Eph}, extended suitably to the infinite dimensional context, could be very appropriate to prove such type of results. 
\medskip

We sketch the main ideas: Let $\YY=\YY(f)\subset \AAA^m$ be \an \arquile variety defined by some power series vector $f(\t,\y)\in\BB^\kk$, and let $y=y(\t)\in \YY$ be a point of $\YY$. We say that $\YY$ is {\it analytically trivial} at $y$ if the completed local ring $\wh\O_{\YY,y}$ is isomorphic to the factor ring $\kkk[[\u_\infty]]/J$ of a formal power series ring $\kkk[[\u_\infty]]=\kkk[[\u_1,\u_2,...]]$ in countably many variables $\u_j$ by an ideal $J$ which can be generated by series not depending on the first variable $\u_1$. This means that the formal neighborhood of $y$ in $\YY$ is isomorphic to a cartesian product of the formal neighborhood of $\A^1_\kkk$ at $0$ with the formal neighborhood of some other \arquile variety at a certain point.\medskip

To simplify the notation, we assume that $y=0\in\AAA^m$ is the zero-vector, so all coefficients $y_{ij}$ are zero for all $i$ and $j$. Writing $\wh\O_{\YY,y}$ as a factor ring $\kkk[[\y_{m,\infty}]]/I_\infty$ for some ideal $I_\infty$ of $\wh \O_{\AAA^m,y}=\kkk[[\y_{m,\infty}]]=\kkk[[\y_{ij},\, 1\leq i\leq m,\, j\geq 1]]$, it is equivalent to say that there exist formal power series $\varphi_{ij}(\y_{m,\infty})\in \kkk[[\y_{m,\infty}]]$, for $1\leq i\leq m,\, j\geq 1$, defining a local $\kkk$-algebra isomorphism $\varphi$ of $\kkk[[\y_{m,\infty}]]$ and such that suitable generators of $\varphi(I_\infty)$ do not depend on one of the variables $\y_{ij}$, say, without loss of generality, $\y_{11}$.\medskip

If an arquile variety satisfies the assertion of the Grinberg-Kazhdan-Drinfeld theorem at a point $y$, then it is in particular analytically trivial at that point.\medskip

Assume now that $I_\infty$ is generated by power series $F_\ell\in\kkk[[\y_{m,\infty}]]$, $\ell\in\N$, such that the series $\varphi(F_\ell)$ generate the ideal $\varphi(I_\infty)$.%
%
%
\footnote{ We use here the notion of generating set in the topological sense, allowing {\smallit infinite} linear combi\-nations as long as the result is well defined in $\kkk[[\y_{m,\infty}]]$. This complication is due to the fact that one may have to take for $I_\infty$ the topological closure in $\kkk[[\y_{m,\infty}]]$ of the ideal generated algebraicially, i.e., by {\smallit finite} linear combinations, by the polynomials $F_\ell$ resulting from $f$ by Taylor expansion.}
%
%
Let $G_n$, $n\in\N$, be generators of $\varphi(I)$ not depending on $\y_{11}$. We may then write $G_n=\sum_{\ell\in\N}\, a_{n\ell}\cdot \varphi(F_\ell)$ for suitable power series $a_{n\ell}\in\kkk[[\y_{m,\infty}]]$. This system of equations can be formally derived with respect to $\y_{11}$ and produces as in the finite dimensional case a regular derivation $\Delta$ of $\kkk[[\y_{m,\infty}]]$ which sends all $F_\ell$ to $I_\infty$, cf.~\cite{Eph}. So the existence of such a derivation is a necessary criterion for the analytic triviality (it should also be sufficient, but this is not used here).\med

Let us carry out this in the example of the arc space of a variety $X\subset\A^m_\kkk$ defined by a Brieskorn polynomial $\y_1^{c_1}+\ldots + \y_m^{c_m}=0$ over a field of characteristic $0$, with exponents $c_i\geq 2$. We consider $X_{\infty,0}$ at the constant arc $y(\t)=0\in\A^m_\kkk$. To simplify the notation, we consider only the case $m=2$ and write $\y$ for $\y_1$ and $\z$ for $\y_2$, as well as $c=c_1$, $d=c_2$. The general case $m\geq 2$ goes analogously. From the equation $\y^c+\z^d=0$ we get the polynomials
$$\ds F_\ell(\y_\infty,\z_\infty)=\sum_{\a\in\N^\N, \abs \a=c,\abs{\abs\a}=\ell}\, \y^\a_\infty+\sum_{\a\in\N^\N, \abs \a=d,\abs{\abs\a}=\ell}\,\z_\infty^\a,$$
where $\a=(\a_1,\a_2,...)$ is a string in $\N^\N$ with only finitely many non-zero entries, $\y^\a_\infty$ stands for the monomial $\prod_i \y_i^{\a_i}$ and where $\abs\a=\sum_i\a_i$ and $\abs{\abs\a}=\sum_i i\cdot \a_i$ denote the total and weighted degrees of $\y^\a$. Denoting by $e_i=(0,...,0,1,0,....)\in\N^\N$ the $i$-th basis vector, we get for all $i\geq 1$
$$\ds \6_{\y_i}F_\ell(\y_\infty,\z_\infty)=\sum_{\a\in\N^\N, \abs \a=c,\abs{\abs\a}=\ell}\, \a_i\cdot \y^{\a-e_i}_\infty.$$
Now observe that for $\ell$ fixed, the monomials $\y^{\a-e_i}$ appearing in the sums are all different for varying $\a$ and $i$. Indeed, $\a-e_i=\b-e_j$ implies $\abs{\abs\a}=\abs{\abs\b}+i-j=\ell=\abs{\abs \b}$, which is only possible if $i=j$; this in turn implies that $\a=\b$. But the monomials in $\kkk[\y_1,\y_2,\ldots]$ are clearly $\kkk$-linearly independent, and hence the same holds for the collection of partial derivatives $\6_{\y_i}F_\ell$ of $F_\ell$, for each fixed $\ell$ (this uses characteristic $0$). Comparing degrees, we conclude that no (finite or infinite) $\kkk$-linear combination of the derivatives $\6_{\y_i}F_\ell$ belongs to $I_\infty$. The same argument applies to the derivative with respect to the variables $\z_i$. Comparing again degrees, one sees that there is actually no regular derivation $\Delta$ of $\kkk[[\y_\infty,\z_\infty]]$ sending all $F_\ell$ into $I_\infty$. By the prospective triviality criterion from above we should then be able to conclude that the formal neighborhood $\wt X_{\infty,0}$ of $X_{\infty,0}$ at $0$ is not analytically trivial. In this case it cannot admit a factorization as in \cite{GK,Dr}.


\vskip 1.2cm 
Faculty of Mathematics\par
University of Vienna, Austria\par
herwig.hauser@univie.ac.at\par
sebastian.woblistin@univie.ac.at\par


\newpage

\section*{Table of symbols}

\begin{table}[ht]
\begin{tabular}{llllll}

$\t$ & a single variable;\\

$\y=(\y_1,...,\y_m)$ & a system of variables;\\


$y=y(\t)=(y_1(\t),...,y_m(\t))$ & a vector of power series;\\

$y_{ij}$,\hs .3cm  $1\leq i\leq m,\, j\geq 1$ &  coefficients of $y$;\\

$\AA$ & one of the rings $\kkk[[\t]]$, $\kkk\{\t\}$, $\kkk\langle\t\rangle$; \\

$\AAA=\AA\cap\k[[\t]]$ & one of the rings $\k[[\t]]$, $\k\{\t\}$, $\k\{\t\}_s$, $\k\langle\t\rangle$  \\ & of power series without constant term;  \\

$\langle \t\rangle^d= \t^d\cdot\AAA$  &  the ideal of $\AAA$ generated by $\t^d$;   \\

$\BB$ & one of the rings $\kkk[[\t,\y]]$, $\kkk\{\t,\y\}$, $\kkk\langle\t,\y\rangle$; \\

$\BBB=\BB\cap\k[[\t,\y]]$ & one of the rings $\k[[\t,\y]]$, $\k\{\t,\y\}$, $\k\{\t,\y\}_s$, $\k\langle\t,\y\rangle$  \\ & of power series without constant term;  \\

$\kkk[\y_{m,\infty}]=\kkk[\y_{ij},\, 1\leq i\leq m,\, j\geq 1]$   &   the polynomial ring in countably many variables $\y_{ij}$;  \\

$f(\t,\y)=(f_1(\t,\y),...,f_\kk(\t,\y))\in\BB^\kk$  &   a vector of formal power series;  \\

$I\subset \BB$   &  the ideal generated by the components $f_i$ of $f$;   \\

$\6_\y f\in\BB^{\kk\times m}$   &  the relative Jacobian matrix of $f$ with respect to $\y$;\   \\

$f(y)=f(\t,y(\t))$   &  the evaluation of $f$ at a vector $y=y(\t)\in\AAA^m$;   \\

$\ff:\AAA^m\map\AA^\kk,\, y\map f(y)= f(\t,y(\t))$&  the \arquile map induced by $f\in\BB^k$ with $f(0)=0$;\\

$\YY(f)=\{y\in\AAA^m,\, f(y) =0\}$   &  the \arquile subvariety of $\AAA^m$ defined by the vector $f$;   \\

$\YY(I) =\{y\in\AAA^m,\, f(y) =0,\, f\in I\}$& the \arquile subvariety of $\AAA^m$   defined by the ideal $I$;   \\

$\Reg(\YY)$, $\Sing(\YY)$ & the \arquile regular and singular loci of $\YY=\YY(I)$;\\

$g\in\kkk[[\t,\y]]$   &   a suitable $(\kk\times\kk)$-minor of $\6_\y f$;  \\

$\langle g\rangle=g\cdot \BBB$   &  the ideal generated by $g$ in $\BB$;   \\

$\SS_d=\{y\in\AAA^m,\, \ord\, g(y) =d\}$&  cofinite \ttextile locally closed stratum, \\& for given $g\in\kkk[[\t,\y]]$ and $d\in\N$; \\ 

$\GG=\langle g_1\rangle\times \ldots\times \langle g_m\rangle\subset \BBB^m

$&  the $\k[[\t,\y]]$-submodule associated to $f$ as a \\& division module, for given 
$g_i\in\kkk[[\t,\y]]$; \\ 

$\GG(y)\subset\AAA^m$&  the $\AAA$-submodule of $\AAA^m$ generated by the \\&  evaluations of elements of $\GG$ at a vector $y\in\AAA^m$; \\ 

$\dd = (d_1,...,d_m)\in\N^m$&  an $m$-tuple of prescribed orders; \\ 

$\SS_\dd =\{y\in\AAA^m,\, \ord\, g_i(y) =d_i$& cofinite \ttextile locally closed stratum, \\
\hskip .8cm for all $i=1,...,m\}$&    for given $g_i\in\kkk[[\t,\y]]$ and $\dd\in\N^m$; \\ 

$D=\diag(g_1,...,g_m)\in \kkk[[\t,\y]]^{m\times m}$&   diagonal matrix with entries $g_1,...,g_m$; \\ 

$\VV_\dd =\GG(y)=\langle\t^{d_1}\rangle\times \ldots\times \langle\t^{d_m}\rangle
$&   the $\AAA$-submodule of $\AAA^m$ generated by the  \\& evaluations of a division module $\GG=\langle g_1\rangle\times \ldots\times \langle g_m\rangle$; \\ 

$\RR_\dd=\k[\t]_{\leq d_1}\times\ldots\times \k[\t]_{\leq d_m}$&   the space of remainders of division by $\langle\t^{d_1}\rangle\times \ldots\times \langle\t^{d_m}\rangle$;   \\ 

$\ZZ_\dd =\SS_\dd  \cap\RR_\dd $;&  stratum of remainders $\zz$ with $\ord\, g_i(\zz)=d_i$;   \\ 
 
$\ZZ_\dd ^*=\{\zz\in\ZZ_\dd ,\, f(\zz)\in\6_\y f(\zz)\cdot \VV_\dd \}$ & finite dimensional factor in Theorem \ref{factorization}; \\ 

$\psi_\dd:\SS_\dd \map \ZZ_\dd \times\VV_\dd ,\, y=\vv+\zz\map (\zz,\vv)$ &  decomposition isomorphism given by division \\& of $y$ by $\GG(y)$ in Proposition \ref{isomorphisms}; \\ 

$\phi_\dd: \ZZ_\dd \times\VV_\dd \map \ZZ_\dd \times\VV_\dd ,$ &automorphism in Proposition \ref{isomorphisms};\\

\hskip .8cm $ (\zz,\vv)\map  (\DD_g(\zz)\cdot [a+p(\zz,a)],\zz)$&   \\ 

$\chi_\dd=(\phi_\dd\circ\psi_\dd)^{-1}:\ZZ_\dd \times \VV_\dd \map\SS_\dd,$ & linearizing isomorphism of Theorems \ref{generallinearization} and \ref{fibration}; \\
\hs.8cm $ y=\vv+\zz=\DD_g(\zz)\cdot a+\zz\map$&\\
\hskip 1.4cm $ \DD_g(\zz)\cdot [a+p(\zz,a)]+\zz$&  \\ 

$\Phi_\dd :\SS_\dd \map \ZZ_\dd \times \VV_\dd $&   linearizing isomorphism of Theorem \ref{factorization}.

\end{tabular}
\end{table}

\end{document}